\documentclass{amsart}

\usepackage{amsmath,amssymb,amsthm,mathrsfs}
\usepackage{here}

\theoremstyle{definition}

\numberwithin{equation}{section}

\newtheorem{theo}{Theorem}[section]
\newtheorem{defi}[theo]{Definition}
\newtheorem{lemm}[theo]{Lemma}

\newtheorem{cor}[theo]{Corollary}
\newtheorem{remk}[theo]{Remark}

\renewcommand{\a}{\alpha}
\newcommand{\oa}{\overline{\alpha}}
\renewcommand{\b}{\beta}
\newcommand{\ob}{\overline{\beta}}
\newcommand{\g}{\gamma}
\newcommand{\og}{\overline{\gamma}}
\newcommand{\e}{\varepsilon}

\newcommand{\n}{\nu}
\newcommand{\on}{\overline{\nu}}
\newcommand{\m}{\mu}
\newcommand{\om}{\overline{\mu}}

\newcommand{\oeta}{\overline{\eta}}

\newcommand{\gc}[1]{g^{\circ} ( #1 )}   
\newcommand{\p}[2]{ (#1)_{#2}}       
\newcommand{\pc}[2]{(#1)_{#2}^{\circ}}

\newcommand{\gh}[3]{
{_{2}F_{1}}\left(
\begin{matrix}
#1 \\
#2
\end{matrix}
; #3
\right)
}  

\newcommand{\kF}[9]{F^{#1}_{#2}\left({#3 \atop #4};{#5 \atop #6};{#7 \atop #8};#9\right)}

\newcommand{\bd}[1]{\boldsymbol{#1} }
\newcommand{\ol}[1]{\overline{#1}}

\title[Kamp\'e de F\'eriet hypergeometric functions over finite fields]{Kamp\'e de F\'eriet hypergeometric functions over finite fields}
\author[R. Ito]{Ryojun Ito}
\author[S. Kumabe]{Satoshi Kumabe}
\author[N. Akio]{Akio Nakagawa}
\author[Y. Nemoto]{Yusuke Nemoto}

\address[Ryojun Ito]{General Education Division, Oshima National College of Maritime Technology, 1091-1, Komatsu, Suo-Oshimacho, Oshima-gun, Yamaguchi, 742-2193, Japan}
\email{ito.ryojun\_math@icloud.com}

\address[Satoshi Kumabe]{Joint Graduate School of Mathematics for Innovation, Kyushu University, 744, Motooka, Nishi-ku, Fukuoka, 819-0395,  Japan}
\email{kuma511ssk@gmail.com}

\address[Akio Nakagawa]{Department of Mathematics and Informatics, Graduate School of Science, Chiba University, 1-33, Yayoicho,  Inage-ku,  Chiba,  263-8522, Japan}
\email{akio.nakagawa.math@icloud.com}

\address[Yusuke Nemoto]{Department of Mathematics and Informatics, Graduate School of Science, Chiba University, 1-33, Yayoicho,  Inage-ku,  Chiba,  263-8522, Japan}
\email{y-nemoto@waseda.jp}

\keywords{character sums,  hypergeometric functions,  reduction formulas, summation formulas}
\subjclass{11L05, 11T24, 33C70, 33C90}

\begin{document}
\maketitle

\begin{abstract}
Kamp\'e de F\'eriet hypergeometric functions  are two-variable hypergeometric functions, which are a generalization of Appell's functions.  It is known that they satisfy many reduction and summation formulas. 
In this paper, we define Kamp\'e de F\'eriet hypergeometric functions over finite fields and show analogous formulas. 
\end{abstract}

\section{Introduction}

Over the complex numbers, Kamp\'e de F\'eriet hypergeometric functions \cite{appkampe, sk} are defined by  
\begin{align*}
F_{A^{\prime};B^{\prime};C^{\prime}}^{A;B;C}\left(\left.
\begin{matrix}
\mathbf{a}  \\
\mathbf{a}^{\prime}
\end{matrix}
; 
\begin{matrix}
\mathbf{b} \\
\mathbf{b}^{\prime}
\end{matrix}
;
\begin{matrix}
\mathbf{c} \\
\mathbf{c}^{\prime}
\end{matrix}
\right| x, y
\right) 
= 
\sum_{m,n=0}^{\infty} 
\frac{(\mathbf{a})_{m+n} (\mathbf{b})_{m} (\mathbf{c})_{n}  }{(\mathbf{a}^{\prime})_{m+n} (\mathbf{b}^{\prime})_{m} (\mathbf{c}^{\prime})_{n}  } \frac{x^{m}y^{n}}{(1)_{m}(1)_{n}},
\end{align*}
where $\mathbf{a} = (a_{1}, \dots, a_{A}), \dots,  \mathbf{c}^{\prime} = (c_{1}^{\prime}, \dots, c_{C^{\prime}}^{\prime})$ are sequences of complex parameters 
with $a_{i}^{\prime},  b_{j}^{\prime},  c_{k}^{\prime}\not \in \mathbb{Z}_{\leq 0}$, 
and $(\mathbf{a})_{n} := \prod_{i=1}^{A} (a_{i})_{n}$ denotes the product of the Pochhammer symbol 
$(a)_{n} := \Gamma (a+n)/ \Gamma (a)$. 
These are a generalization of generalized hypergeometric functions ${}_{p}F_{q}(x)$ and the Appell functions $F_{1}$, $F_{2}$, $F_{3}$ and $F_{4}$. For example,  we have
\begin{align*}
&F_{q; 0 ; 0}^{p;0;0}\left(\left.
\begin{matrix}
a_{1}, \dots, a_{p}  \\
a_{1}^{\prime}, \dots, a_{q}^{\prime}
\end{matrix}
; 
\begin{matrix}
\hspace{1mm} \\
\hspace{1mm}
\end{matrix}
;
\begin{matrix}
\hspace{1mm} \\
\hspace{1mm}
\end{matrix}
\right| \frac{x}{2}, \frac{x}{2}
\right)
= {_{p}}F_{q}\left( 
\begin{matrix}
a_{1}, \dots, a_{p}  \\
a_{1}^{\prime}, \dots, a_{q}^{\prime}
\end{matrix}
; x
\right), \\
&F_{1;0;0}^{1;1;1}\left( \left.
\begin{matrix}
a \\
a^{\prime}
\end{matrix}
;
\begin{matrix}
b \\
\hspace{1mm}
\end{matrix}
;
\begin{matrix}
c \\
\hspace{1mm}
\end{matrix}
\right| 
x,y
\right)
=F_1 
\left( \left. 
a; b,c;a^{\prime}; x,y
\right.\right).
\end{align*}

Over finite fields, there are several definitions of hypergeometric functions of one variable due to Koblitz \cite{koblitz}, Greene \cite{greene}, McCarthy \cite{mc}, Fuselier-Long-Ramakrishna-Swisher-Tu \cite{flrst}, Katz \cite{katz}, and Otsubo \cite{otsubo}. 
As for two-variable functions, Li-Li-Mao \cite{llm} and He \cite{he2} define the Appell functions $F_{1}$, 
He-Li-Zhang \cite{hlz} and Ma \cite{ma} define functions $F_{2}$, He \cite{he1} defines functions $F_{3}$, Tripathi-Barman \cite{tb} define functions $F_{4}$,  and     
Tripathi-Saikia-Barman \cite{tsb} and Otsubo \cite{otsubo} define functions $F_{1}, \dots, F_{4}$. 
More generally, Otsubo also defines the Lauricella functions $F_{A}^{(n)}, \dots, F_{D}^{(n)}$, 
which are a multi-variable generalization of Appell's functions.

Similarly to generalized hypergeometric functions and Appell's functions, Kamp\'e de F\'eriet hypergeometric functions over the complex numbers satisfy many reduction and summation formulas. 
For example, Liu-Wang \cite{liuwang} prove the following reduction formulas: 
\begin{align}
\begin{split}
F_{1:0:0}^{1:1:1} \left( \left. 
\begin{matrix}
a \\
b
\end{matrix}
;
\begin{matrix}
c  \\
\hspace{1mm}
\end{matrix}
;
\begin{matrix}
b-c \\
\hspace{1mm}
\end{matrix} 
\right| 
x, y
\right)  
&=
(1-x)^{-a} {}_{2}F_{1} \left( 
\begin{matrix}
a , b-c  \\
b
\end{matrix}
; \cfrac{x-y}{x-1}
\right). 
\end{split} \label{cor2.3} \\
F_{1:0:1}^{1:1:2} \left( \left. 
\begin{matrix}
a \\
b
\end{matrix}
;
\begin{matrix}
c  \\
\hspace{1mm}
\end{matrix}
;
\begin{matrix}
b-c, d \\
e
\end{matrix} 
\right| 
x, x
\right)  
&=
(1-x)^{-a} {}_{3}F_{2} \left( 
\begin{matrix}
a , b-c , e-d \\
b, e
\end{matrix}
; \cfrac{x}{x-1}
\right).  \label{cor2.4}\\
\begin{split}
F_{1:0:0}^{1:1:1} \left( \left. 
\begin{matrix}
a \\
b
\end{matrix}
;
\begin{matrix}
c  \\
\hspace{1mm}
\end{matrix}
;
\begin{matrix}
b-c \\
\hspace{1mm}
\end{matrix} 
\right| 
x, x
\right)  
&= (1-x)^{-a}.
\end{split} \label{cor2.10} \\
F_{1:0:1}^{1:1:2} \left( \left. 
\begin{matrix}
a \\
b
\end{matrix}
;
\begin{matrix}
c  \\
\hspace{1mm}
\end{matrix}
;
\begin{matrix}
b-c, d \\
b+d
\end{matrix} 
\right| 
x, x
\right)  
&=(1-x)^{b-a-c} {}_{2}F_{1} \left( 
\begin{matrix}
b-c,  b+d-a \\
b+d
\end{matrix}
; x
\right).  \label{cor2.11} \\
F_{1:0:1}^{1:1:2} \left( \left. 
\begin{matrix}
a \\
b
\end{matrix}
;
\begin{matrix}
c  \\
\hspace{1mm}
\end{matrix}
;
\begin{matrix}
b-c, \frac{a-b}{2} \\
1+\frac{a+b}{2}
\end{matrix} 
\right| 
x, x
\right)  
&=(1-x)^{b-a-c} {}_{3}F_{2} \left( 
\begin{matrix}
b-c, 1+ \frac{b}{2}, \frac{b-a}{2} \\
\frac{b}{2}, 1 + \frac{a+b}{2}
\end{matrix}
; x
\right).  \label{cor2.13} \\
F_{1:0:0}^{0:2:2} \left( \left. 
\begin{matrix}
\hspace{1mm} \\
b
\end{matrix}
;
\begin{matrix}
a, c  \\
\hspace{1mm}
\end{matrix}
;
\begin{matrix}
b-a, b-c\\
\hspace{1mm}
\end{matrix} 
\right| 
x, y
\right)  
&=
(1-x)^{b-a-c} {}_{2}F_{1} \left( 
\begin{matrix}
b-a,  b-c \\
b
\end{matrix}
; x+y-xy
\right). \label{cor3.2} \\
F_{1:0:1}^{0:2:3} \left( \left. 
\begin{matrix}
\hspace{1mm} \\
b
\end{matrix}
;
\begin{matrix}
a,c  \\
\hspace{1mm}
\end{matrix}
;
\begin{matrix}
b-a, b-c, d \\
e
\end{matrix} 
\right| 
x, \cfrac{x}{x-1}
\right)  
&=
(1-x)^{b-a-c} {}_{3}F_{2} \left( 
\begin{matrix}
b-a, b-c,  e-d \\
b, e
\end{matrix}
; x
\right). \label{cor3.3} \\
\begin{split}
F_{1:0:0}^{0:2:2} \left( \left. 
\begin{matrix}
\hspace{1mm} \\
b
\end{matrix}
;
\begin{matrix}
a, c  \\
\hspace{1mm}
\end{matrix}
;
\begin{matrix}
b-c, d \\
\hspace{1mm}
\end{matrix} 
\right| 
x, \cfrac{x}{x-1}
\right)  
&=(1-x)^{-\a} {}_{2}F_{1} \left( 
\begin{matrix}
b-c, a+d   \\
b
\end{matrix}
; \cfrac{x}{x-1}
\right).
\end{split}  \label{cor3.6} \\
F_{1:0:1}^{0:2:3} \left( \left. 
\begin{matrix}
\hspace{1mm} \\
b
\end{matrix}
;
\begin{matrix}
a,c  \\
\hspace{1mm}
\end{matrix}
;
\begin{matrix}
b-c, d,  e \\
a+d+e
\end{matrix} 
\right| 
x, \cfrac{x}{x-1}
\right)  
&=
(1-x)^{-a} {}_{3}F_{2} \left( 
\begin{matrix}
b-c, a+d, a+e  \\
b, a+d+ e
\end{matrix}
; \cfrac{x}{x-1}
\right). \label{cor3.7} \\
F_{1:0:1}^{0:2:3} \left( \left. 
\begin{matrix}
\hspace{1mm} \\
b
\end{matrix}
;
\begin{matrix}
a,c  \\
\hspace{1mm}
\end{matrix}
;
\begin{matrix}
b-c, d,  -\frac{a}{2} \\
1 + \frac{a}{2} + d
\end{matrix} 
\right| 
x, \cfrac{x}{x-1}
\right)  
&=
(1-x)^{-a} {}_{4}F_{3} \left( 
\begin{matrix}
a+d, 1 + \frac{a+d}{2}, \frac{a}{2}, b-c \\
\frac{a+d}{2}, 1 + \frac{a}{2} +d, b
\end{matrix}
; \cfrac{x}{x-1}
\right). \label{cor3.9} \\
F_{1:0:1}^{2:0:1} \left( \left. 
\begin{matrix}
a, d \\
b
\end{matrix}
;
\begin{matrix}
\hspace{1mm} \\
\hspace{1mm}
\end{matrix}
;
\begin{matrix}
c \\
c+b
\end{matrix} 
\right| 
x, -x
\right)  
&=
(1-x)^{-a} {}_{2}F_{1} \left( 
\begin{matrix}
a, b+c-d \\
b+c
\end{matrix}
; \cfrac{x}{x-1}
\right). \label{cor4.2} \\
F_{1:0:1}^{2:0:1} \left( \left. 
\begin{matrix}
a, c \\
b
\end{matrix}
;
\begin{matrix}
\hspace{1mm} \\
\hspace{1mm}
\end{matrix}
;
\begin{matrix}
\frac{c-b}{2} \\
1+\frac{c+b}{2}
\end{matrix} 
\right| 
x, -x
\right)  
&=
(1-x)^{-a} {}_{3}F_{2} \left( 
\begin{matrix}
a, 1 + \frac{b}{2}, \frac{b-c}{2} \\
\frac{b}{2}, 1 + \frac{c+b}{2}
\end{matrix}
; \cfrac{x}{x-1}
\right). \label{cor4.4} 
\end{align}
Moreover, they also show summation formulas for Kamp\'e de F\'eriet hypergeometric functions by specializing parameters and arguments in these formulas and using summation formulas for functions ${}_{2}F_{1}$. 
For example, by \eqref{cor2.11} and the Kummer summation formula, they obtain \cite[Corollary 5.1]{liuwang}
\begin{align*}
F_{1:0:1}^{1:1:2} \left( \left. 
\begin{matrix}
a\\
b
\end{matrix}
;
\begin{matrix}
c \\
\hspace{1mm}
\end{matrix}
;
\begin{matrix}
b-c, \frac{1+a -b-c}{2}  \\
\frac{1 + a + b -c}{2}
\end{matrix} 
\right| 
-1,-1
\right)  
&= 2^{b-a-c}
\frac{\Gamma \left( \frac{1+a+b-c}{2} \right) \Gamma \left( 1 + \frac{b-c}{2} \right)}{\Gamma(1+b-c)\Gamma \left( \frac{1+a}{2} \right)}.
\end{align*}
We remark that the formulas \eqref{cor2.3} and \eqref{cor2.4} are originally obtained by Bailey \cite{bailey} and Cvijovi\'c-Miller \cite{cm}, respectively.

In this paper, we define Kamp\'e de F\'eriet hypergeometric functions over a finite field $\kappa$ with $q$ elements. 
As well as the complex case, our definition is a unified generalization of the Appell functions $F_{1}, \dots, F_{4}$ due to Tripathi-Saikia-Barman and Otsubo. 
Note that their definitions are different, but one can show that they coincide. 
Here we use Otsubo's notations. See \S 2 for details.  
A finite field analogue of the Pochhammer symbol and its variant are defined by 
the quotients of Gauss sums
\begin{align*}
\p{\a}{\n} = \frac{g(\a\n)}{g(\a)},  \hspace{6mm}
\pc{\a}{\n} = \frac{\gc{\a\n}}{ \gc{\a} },
\end{align*}
where $\a$, $\n \in \widehat{\kappa^{*}} := \mathrm{Hom}(\kappa^{*}, \mathbb{C}^{*})$. 
Then we define a Kamp\'e de F\'eriet hypergeometric function over $\kappa$ by the 
double sum
\begin{align*}
F_{A^{\prime}:  B^{\prime}:  C^{\prime}}^{A : B : C} \left( \left. 
\begin{matrix}
\bd{\a} \\
\bd{\a}^{\prime}
\end{matrix}
;
\begin{matrix}
\bd{\b}  \\
\bd{\b}^{\prime}
\end{matrix}
;
\begin{matrix}
\bd{\g} \\
\bd{\g}^{\prime}
\end{matrix} 
\right| 
x,  y
\right) 
= \frac{1}{(1-q)^{2}} \sum_{\n,    \m \in \widehat{\kappa^{*}}} 
\frac{ \p{\bd{\a}}{\n\m} \p{\bd{\b}}{\n}  \p{\bd{\g}}{\m}     }{\pc{ \bd{\a}^{\prime}  }{\n\m} \pc{\bd{\b}^{\prime} }{\n}  \pc{\bd{\g}^{\prime} }{\m}   \pc{\e}{\n} \pc{\e}{\m} }  \n(x) \m(y),
\end{align*}
where $\varepsilon$ is the trivial character and $\bd{\a}$,  $\bd{\b}$,  $\bd{\g}$,  $\bd{\a}^{\prime}$,  
$\bd{\b}^{\prime}$,  $\bd{\g}^{\prime}$ are elements of the free abelian monoid over $\widehat{\kappa^{*}}$. 
We will give finite field analogues of the reduction formulas above and, as their applications, show summation formulas. 
Our strategy is to apply Liu-Wang's proof in the case of finite fields. 
First, we show certain transformation formulas for hypergeometric double sums by slightly modifying
the  Euler and Pfaff transformation formulas for functions ${}_{2}F_{1}$ over $\kappa$. 
These formulas yield several reduction formulas for Kamp\'e de F\'eriet hypergeometric functions.  
Then, by specializing parameters and arguments and applying summation formulas for functions ${}_{2}F_{1}$, we obtain summation formulas for them.

We remark that the finite field analogues of the remaining formulas in \cite{liuwang} can be reduced to our formulas for the following reason. 
For any parameters $a$, $a+m \in \mathbb{C}$ with $m \in \mathbb{Z}$, their finite field analogues agree with the same character since the finite field analogue of the parameter $1$ is the trivial character. 
Hence, the finite field analogues of hypergeometric functions in the remaining formulas contain common parameters in the denominators and the numerators. 
Then, the finite field analogues of the remaining formulas agree with our results by the cancellation formula \eqref{cancellation} below.

To explain the strong similarities between complex and finite Kamp\'e de F\'eriet hypergeometric functions, 
one may expect the existence of algebraic varieties over number fields whose periods are 
those functions over $\mathbb{C}$ and whose point counting over a residue field $\kappa$ leads to 
those functions over $\kappa$. 
Actually, for generalized hypergeometric functions ${}_{n+1}F_{n}$ and the Appell-Lauricella functions, such varieties are known. See Koblitz \cite{koblitz} and the third author \cite{nakagawa}, respectively. 
As far as the authors know, integral representations of Kamp\'e de F\'eriet hypergeometric functions 
over $\mathbb{C}$, hence the corresponding algebraic varieties, are only known in some special cases. 
One may be able to establish sum representations of Kamp\'e de F\'eriet hypergeometric functions over 
$\kappa$ to find suitable algebraic varieties and then derive those functions over $\mathbb{C}$ as periods of the varieties.

This paper is constructed as follows:  
In \S 2, after introducing Otsubo's notations, we define Kamp\'e de F\'eriet hypergeometric functions over finite fields.  
Then we list some formulas for generalized hypergeometric functions over finite fields, which will be used in our proofs.  
In \S 3, we show transformation formulas for hypergeometric sums and some reduction and summation formulas for Kamp\'e de F\'eriet hypergeometric functions over finite fields.

\section{Definitions}
Let us recall the notations used in \cite{otsubo} to define Kamp\'e de F\'eriet hypergeometric functions over finite fields. 
Let $\kappa$ be a finite field of characteristic $p$ with $q$ elements.  
We define the parameter set $P$ by the free abelian monoid over $\widehat{\kappa^{*}} := \mathrm{Hom}(\kappa^{*}, \mathbb{C}^{*})$.  For any $\varphi \in \widehat{\kappa^{*}}$,  we set  $\varphi (0) = 0$.  
Let $\e \in \widehat{\kappa^{*}}$ be the trivial character and $\delta \colon \widehat{\kappa^{*}} \rightarrow \mathbb{C}$
the characteristic function of $\e$,  that is,  $\delta (\e) =1$ and $\delta (\varphi) = 0$ for $\varphi \neq \e$. 
Throughout this paper,  we fix a non-trivial additive character $\psi \in \widehat{\kappa}$.  
For a multiplicative character $\varphi \in \widehat{\kappa^{*}}$,  the Gauss sum, which is a finite field analogue of the gamma function, and its variant are defined by 
\begin{align*}
g(\varphi) = - \sum_{x \in \kappa} \varphi (x) \psi (x)  ,  \hspace{6mm} \gc{\varphi} = q^{\delta (\varphi)} g(\varphi).
\end{align*}
Then we define the Pochhammer symbol and its variant  by 
\begin{align*}
\p{\varphi}{\n} = \frac{g(\varphi\n)}{g(\varphi)},  \hspace{6mm}
\pc{\varphi}{\n} = \frac{\gc{\varphi\n}}{\gc{\varphi}},
\end{align*}
where $\varphi, \n \in \widehat{\kappa^{*}}$. 

Let $\deg \colon P \rightarrow \mathbb{Z}$ be the degree map. Let $(\hspace{1mm},   \hspace{1mm}) \colon  P \times P \rightarrow \mathbb{Z}_{\geq 0}$ be the symmetric paring 
extending $(\a,  \b) = \delta (\a \ob)$.  
We extend the Pochhammer symbols to  $\bd{\a} \in P$ by 
\begin{align*}
\p{\bd{\a}}{\n} := \prod_{\a \in P} {\p{\a}{\n}} ^{(\a,  \bd{\a})},   \hspace{5mm}
\pc{\bd{\a}}{\n} := \prod_{\a \in P} {\pc{\a}{\n}} ^{(\a,  \bd{\a})}.  
\end{align*}

With the notations as above, the hypergeometric function $F(\bd{\a}, \bd{\b}; x)$ on $\kappa$ with parameters $\bd{\a}$, $\bd{\b} \in P$ and a variable $x \in \kappa$ is defined by \cite[Definition 2.7]{otsubo}
\begin{align*}
F(\bd{\a}, \bd{\b}; x)
= \frac{1}{1-q} \sum_{\n \in \widehat{\kappa^{*}}} \frac{\p{\bd{\a}}{\n}}{\pc{\bd{\b}}{\n}} \n (x).
\end{align*}
When $\bd{\a} = \a_{1} + \cdots + \a_{A}$, $\bd{\b} = \e + \b_{1}+ \cdots + \b_{B}$, we also write 
\begin{align*}
F(\bd{\a}, \bd{\b}; x)
= {}_{A}F_{B} \left( 
\begin{matrix}
\a_{1}, \dots, \a_{A} \\
\b_{1}, \dots, \b_{B}
\end{matrix}
; x
\right).
\end{align*}

Now,  we define Kamp\'e de F\'eriet hypergeometric functions over $\kappa$.

\begin{defi}
We define the Kamp\'e de F\'eriet hypergeometric function on $\kappa^{2}$ with parameters $\bd{\a}$,  $\bd{\b}$,  $\bd{\g}$,  $\bd{\a}^{\prime}$, $\bd{\b}^{\prime}$,  $\bd{\g}^{\prime} \in P$ 
and variables $x$, $y \in \kappa$ by
\begin{align*}
F_{A^{\prime}:  B^{\prime}:  C^{\prime}}^{A : B : C} \left( \left. 
\begin{matrix}
\bd{\a} \\
\bd{\a}^{\prime}
\end{matrix}
;
\begin{matrix}
\bd{\b}  \\
\bd{\b}^{\prime}
\end{matrix}
;
\begin{matrix}
\bd{\g} \\
\bd{\g}^{\prime}
\end{matrix} 
\right| 
x,  y
\right) 
= \frac{1}{(1-q)^{2}} \sum_{\n,    \m \in \widehat{\kappa^{*}}} 
\frac{ \p{\bd{\a}}{\n\m} \p{\bd{\b}}{\n}  \p{\bd{\g}}{\m}     }{\pc{ \bd{\a}^{\prime}  }{\n\m} \pc{\bd{\b}^{\prime} }{\n}  \pc{\bd{\g}^{\prime} }{\m}   \pc{\e}{\n} \pc{\e}{\m} }  \n(x) \m(y) .
\end{align*}
Here $A, \dots, C^{\prime}$ denote the degrees of parameters $\bd{\a}, \dots, \bd{\g}^{\prime}$ respectively. 
\end{defi}
By definition, we have $F_{A^{\prime}:  B^{\prime}:  C^{\prime}}^{A : B : C} (x,y) =0$ when $xy=0$.  
In general, this function takes values in $\mathbb{Q}(\mu_{p(q-1)})$. 
When $\deg (\bd{\a}) + \deg (\bd{\b}) = \deg (\bd{\a}^{\prime}) + \deg (\bd{\b}^{\prime}) + 1$ and $\deg (\bd{\a}) + \deg (\bd{\g}) = \deg (\bd{\a}^{\prime}) + \deg (\bd{\g}^{\prime}) + 1$, 
one can show that 
its values are in $\mathbb{Q}(\mu_{q-1})$ and do not depend on the choice of an additive character $\psi$ by \cite[Lemma 2.4 (iii)]{otsubo}.

We list some formulas which will be used in our proofs below.

The finite field analogue of the reflection formula $\Gamma (s) \Gamma (1-s) = \pi /  \sin \pi s$ (cf.  \cite[Proposition 2.2 (iii)]{otsubo}): 
\begin{align}
g(\varphi) \gc{\overline{\varphi}} = \varphi (-1)q,  \label{reflection}
\end{align}
where $\overline{\varphi} = \varphi^{-1}$.  This implies 
\begin{align}
\p{\a}{\n} \pc{\oa}{\overline{\n}} = \n (-1).   \label{prelection}
\end{align}

The  transitivity of Pochhammer symbols (cf. \cite[Lemma 2.4 (i)]{otsubo}):
\begin{align}
\p{\a}{\n\m}  = \p{\a}{\n} \p{\a\n}{\m},  \hspace{6mm} \pc{\a}{\n\m} = \pc{\a}{\n} \pc{\a\n}{\m}.   \label{transitivity}
\end{align}

The shift and cancellation formulas \cite[Proposition 2.9,  Theorem 3.2]{otsubo}: 
\begin{align}
&F(\bd{\a},  \bd{\b};  x) = \frac{\p{\bd{\a}}{\varphi}}{\pc{\bd{\b}}{\varphi}} 
\varphi (x) F(\bd{\a} \varphi,  \bd{\b}\varphi; x), \label{shift}  \\
&F(\bd{\a} + \bd{\g},  \bd{\b} + \bd{\g};  x)
= q^{(\bd{\g},  \e)} \left( F(\bd{\a},  \bd{\b};  x) 
+ q^{-1} \sum_{\n \in \widehat{\kappa^{*}}} \frac{1 -q^{-(\bd{\g},  \n)} }{1-q^{-1}} \frac{\p{\bd{\a}}{\overline{\n}}}{\pc{\bd{\b}}{\overline{\n}}} \overline{\n} (x)  \right)  .  \label{cancellation}
\end{align}
Here we write $\bd{\a}\varphi = \a_{1}\varphi + \cdots + \a_{r} \varphi$ for $\bd{\a} = \a_{1} + \cdots + \a_{r}$.

For $x \in \kappa^{*}$, we have \cite[Corollary 3.4, Example 3.9 (ii)]{otsubo}
\begin{align}
{}_{1}F_{0} \left( 
\begin{matrix}
\m \\
\hspace{1mm}
\end{matrix}
; x
\right)
= \begin{cases}
\displaystyle  \om \left( 1 - x \right) & (\m \neq \e), \\
\displaystyle -q \delta \left( 1 - x \right) +1  & (\m = \e).
\end{cases}
\label{geom}
\end{align}

The finite field analogues of the Euler and Pfaff transformation formulas (cf. \cite[Theorem 3.14 (i), (ii)]{otsubo}): 
\begin{align}
\gh{\a, \b}{\g}{x} = \ol{\a\b}\g (1-x) \gh{\oa\g, \ob\g}{\g}{x}, \label{euler0} \\
\gh{\a, \b}{\g}{x} = \oa (1-x) \gh{\a, \ob\g}{\g}{\frac{x}{x -1}}, \label{pfaff0} 
\end{align}
where  $(\a + \b, \e + \g) =0$ and $x \neq 1$.

The  finite field analogue of the Euler-Gauss summation formula (cf.  \cite[Theorem 4.3]{otsubo}):
\begin{align} 
\gh{\a, \b}{\g}{1} = \frac{\gc{\g} g(\g\ol{\a\b})}{\gc{\g\oa} \gc{\g\ob}} , \label{eulergauss}
\end{align}
where $\a + \b \neq \g + \e$. This implies the following finite field analogue of the Vandermonde formula \cite[Remark 4.4 (ii)]{otsubo}: 
\begin{align}
\gh{\a, \on}{\g}{1} = \cfrac{(\oa\g)_{\n}}{(\g)^{\circ}_{\n}},   \label{Vande}
\end{align}
where $(\a,  \e + \g) =0$ and $\n \in \widehat{\kappa^{*}}$.

Assume $p \neq 2$. Then we have the finite field analogue of the Kummer summation formula (cf. \cite[Theorem 4.6 (ii)]{otsubo}): 
\begin{align}
\gh{\a^{2}, \b}{\a^{2}\ob}{-1} 
= \sum_{{\a^{\prime}}^{2} = \a^{2}} \frac{\gc{\a^{2}\ob} g(\a^{\prime}) }{g(\a^{2}) \gc{\a^{\prime}\ob}}. \label{kummer}
\end{align}
If $(\a^2+\b^2+\a\ob, \e)=0$, by \eqref{pfaff0} and \eqref{kummer}, we have 
\begin{align*}
\gh{\a^2, \b^2}{\a\b}{\dfrac{1}{2}}
= \a (4) \sum_{{\a^{\prime}}^{2} = \a^{2}} \frac{\gc{\a\b}  g(\a^{\prime}) }{ g(\a^{2}) \gc{\a^{\prime} \oa \b } }.
\end{align*}
If we use the duplication formula \cite[Theorem 3.10]{otsubo}
\begin{align*}
g(\a^{2}) = \frac{ \a(4) g(\a) g(\a \phi)}{ g(\phi) }
\end{align*}
and put $\a^{\prime} = \a \chi$ ($\chi^{2} = \e$),  we obtain 
\begin{align}
\gh{\a^2,\b^2}{\a\b}{\dfrac{1}{2}}=\sum_{\chi^2=\e} \dfrac{\gc{\a\b}g(\phi)}{g(\phi\chi\a)g(\chi\b)}, \label{gsecond} 
\end{align}
which is a finite field analogue of the Gauss second summation formula. Here $\phi$ denotes the quadratic character. By similar computation,  we also obtain the following 
finite field analogue of the Bailey summation formula
\begin{align}
\gh{ \a^2,\oa^2 }{ \g^2 }{ \dfrac{1}{2} } = \sum_{ \chi^2=\e }\dfrac{   \gc{\g}\gc{\phi\g} }{   g(\phi\chi\a\g) g(\chi\overline{\a}\g) }, \label{bailey}
\end{align}
if $(\a^2+\a^2\g^2+\a^2\og^2,\e)=0$.

Recall that the function $F(\bd{\a},  \bd{\b};  x)$ is said to be Saalsch\"utzian if $\a_{1} \cdots \a_{d} = \b_{1} \cdots \b_{d}$ 
(see \cite[Definition 4.1]{otsubo}).  
For a Saalsch\"utzian ${}_{3}F_{2}(1)$, we have the finite field analogue of the Saalsch\"utz formula (cf.  \cite[Theorem 4.11]{otsubo}): 
\begin{align}
{}_{3}F_{2} \left( 
\begin{matrix}
\a,  \b,  \g \\
\varphi,  \psi
\end{matrix}
;  1 
\right)
= \frac{\gc{\varphi} g(\a\overline{\psi}) g(\b \overline{\psi})  g(\g \overline{\psi})  }{g(\overline{\psi}) \gc{\oa\varphi}
\gc{\ob\varphi} \gc{\og\varphi} } 
+ \frac{\gc{\varphi} \gc{\psi} }{g(\a) g(\b) g(\g)},  \label{saalschutz}
\end{align}
where $\a\b\g = \varphi\psi$ and $\a + \b + \g \neq \e + \varphi + \psi$.

\section{Reduction and Summation formulas}
In this section, we give some transformation formulas for hypergeometric double sums 
and reduction formulas for Kamp\'e de F\'eriet hypergeometric functions over 
finite fields.  As an application, we also give some summation formulas for them.

\subsection{Formulas for $F_{1: 0:C^{\prime}}^{1:1:C} $}

%%%%%%%%%%%%%%%%%%%%%%%%%%% Ito's part1 %%%%%%%%%%%%%%%%%%%%%%%%%%%%%%%%%%%

Liu-Wang prove the following transformation formula \cite[Theorem 2.1]{liuwang}:
\begin{align*}
&\sum_{i,j=0}^{\infty} \Omega(j) \frac{\p{a}{i+j} \p{c}{i} \p{b - c}{j} }{\p{b}{i+j}} \frac{x^{i} y^{j}}{i!j!} \\
&= (1-x)^{-a} \sum_{n=0}^{\infty} \frac{\p{a}{n} \p{b - c}{n} }{\p{b}{n} n! } \left( \frac{x}{x-1}\right)^{n}
\sum_{j=0}^{n} \Omega (j) \frac{\p{-n}{j}}{j!} \left( \frac{y}{x} \right)^{j}.
\end{align*}
Here $\Omega (j)$ denotes an arbitrary complex sequence.

The following is an analogous formula. 

\begin{theo}\label{first trans}
Assume $(\a + \a\g,  \b) = (\g, \e) = 0$.  Let $f\colon  \widehat{\kappa^{*}} \rightarrow \mathbb{C}$ be a map.  
For any $x, y \in \kappa$ with $x \neq 0,  1$,   we have 
\begin{align*}
&\sum_{\n,\m \in \widehat{\kappa^{*}} } f(\m) \frac{ \p{\a}{\n\m}  \p{\g}{\n} \p{\b\og}{\m} }{\pc{\b}{\n\m}  \pc{\e}{\n} \pc{\e}{\m}} \n(x) \m(y) \\
&= \oa (1-x) \sum_{\eta \in \widehat{\kappa^{*}}} \frac{\p{\a + \b\og}{\eta}   }{\pc{\b + \e }{\eta} } \eta \left( \frac{x}{x-1} \right)
\sum_{\m \in \widehat{\kappa^{*}} } f(\m) \frac{\p{ \oeta }{\m}  }{\pc{\e}{\m}} \m \left( \frac{y}{x} \right)  \\
& \hspace{20mm}  + (1-q)^{2} f(\ob \g) \og(x)\ob\g(y) \frac{\pc{\e}{\b}}{\pc{\oa}{\b}}.
\end{align*}
\end{theo}
\begin{proof}
By \eqref{transitivity} and applying Lemma \ref{pfaff1} below,  the left hand side equals
\begin{align*}
&\sum_{\m\in \widehat{\kappa^{*}} } f(\m) \frac{\p{\a + \b\og }{\m}   }{\pc{\b + \e}{\m}    } \m(y) 
\sum_{\n \in \widehat{\kappa^{*}} } \frac{\p{\a\m  + \g}{\n}  }{\pc{\b\m + \e }{\n}  } \n(x) \\
&= (1-q)\sum_{\m \in \widehat{\kappa^{*}} } f(\m) \frac{\p{\a + \b\og }{\m}   }{\pc{\b + \e}{\m}    } \m(y) \gh{\a\m,  \g}{\b\m}{x} \\
&=(1-q)\sum_{\m \in \widehat{\kappa^{*}}} f(\m)  \frac{\p{\a + \b\og }{\m}   }{\pc{\b + \e}{\m}    } \m(y) 
\left( 
\oa\om (1-x) \gh{\a\m,  \og\b\m}{\b\m}{\frac{x}{x-1}} \right.  \\
&\hspace{40mm} \left.  + \delta (\og \b\m) (1-q) \frac{g(\a\ob)}{g(\a\g\ob) g(\og)} \og(x)
\right) \\
&= \oa (1-x) \sum_{\n,  \m \in \widehat{\kappa^{*}}} f(\m) \frac{\p{\a + \b\og}{\n\m}  }{\pc{\b}{\n\m} \pc{\e}{\n} \pc{\e}{\m} } 
\n \left( \frac{x}{x-1} \right) \m \left( \frac{y}{1-x} \right)  \\
& \hspace{10mm} + (1-q)^{2} f(\ob\g)  \frac{\p{\a + \b\og }{\ob\g}   }{\pc{\b + \e}{\ob\g}    } \ob\g (y)
\cdot \frac{g(\a\ob)}{g(\a\g\ob) g(\og)} \og (x) . 
\end{align*}
If we put $\n\m = \eta$ and use \eqref{prelection}, the first term is equal to
\begin{align*}
\oa (1-x) \sum_{\eta \in \widehat{\kappa^{*}}} \frac{\p{\a + \b\og }{\eta}  }{\pc{\b + \e }{\eta} } \eta \left( \frac{x}{x-1} \right)
\sum_{\m \in \widehat{\kappa^{*}}} f(\m) \frac{ \p{\oeta}{\m} }{\pc{\e}{\m} } \m \left( \frac{y}{x} \right). 
\end{align*}
Finally,  by \eqref{reflection},  the second term coincides with 
\begin{align*}
(1-q)^{2} f(\ob\g)\frac{\pc{\e}{\b}}{\pc{\oa}{\b}} \og(x) \ob\g(y).
\end{align*}
Therefore we obtain the formula.
\end{proof}

\begin{lemm}\label{pfaff1}
Assume $(\a + \a\g,  \b) = (\g, \e) = 0$.  For any $\m \in \widehat{\kappa^{*}}$ and $x \in \kappa \backslash \{ 1 \}$,  we have 
\begin{align*}
\gh{\a\m,  \g}{\b\m}{x} 
&= \oa\om (1-x) \gh{\a\m,  \og\b\m}{\b\m}{\frac{x}{x-1}} \\
& \hspace{20mm} + \delta (\m\b\og) (1-q) \frac{g(\a\ob)}{g(\a\g\ob) g(\og)} \og(x). 
\end{align*}
\end{lemm}
\begin{proof}
The case $x =0$ is clear, hence we may assume $x \neq 0$.
The formula in the case of $(\a\m + \g,  \e + \b\m) =0$, that is, $\m \neq \oa$,  $\ob\g$
follows from \eqref{pfaff0}. We show the remaining cases below.

Consider the case $\m = \oa$.  By \eqref{cancellation},  \eqref{shift} and \eqref{geom}, 
the left hand side agrees with 
\begin{align*}
\gh{\e,  \g}{\b\oa}{x} 
= q F (\g,  \b\oa;  x) + 1
&= q \frac{\p{\g}{\ob\a}}{\pc{\b\oa}{\ob\a}} \ob\a (x) {}_{1}F_{0} \left( 
\begin{matrix}
\a\g\ob \\
\hspace{2mm}
\end{matrix}
; x \right) + 1 \\
&= \frac{g(\a\g\ob) \gc{\oa\b} }{g(\g)} \a\ob(x) \ol{\a\g} \b(1-x) + 1 \\
&= \frac{g(\a\g\ob) g(\oa\b) }{g(\g)} \a\ob(x) \ol{\a\g} \b(1-x) + 1.
\end{align*}
Here we used the assumption $(\a,  \b)=0$ in the last equality. Similarly to the computation above,  
we know that the right hand side is equal to
\begin{align*}
\gh{\e,  \ol{\a\g} \b}{\b\oa}{\frac{x}{x-1}}
&= q F \left( \ol{\a\g}\b,  \b\oa;  \frac{x}{x-1} \right) + 1 \\
&= q \frac{\p{\ol{\a\g} \b }{\ob\a}}{ \pc{\b\oa}{\ob\a} } \ob\a \left( \frac{x}{x-1} \right) 
{}_{1}F_{0} \left( 
\begin{matrix}
\og \\
\hspace{2mm}
\end{matrix}
; \frac{x}{x-1}
\right)  + 1  \\
&= \frac{g (\og) \gc{\b\oa} }{g(\ol{\a\g}\b)}  \ob\a \left( \frac{x}{x-1} \right) \g \left( 1 - \frac{x}{x-1} \right) +1 \\
&= \frac{g(\a\g\ob) g(\oa\b) }{g(\g)} \a\ob(x) \ol{\a\g} \b(1-x) + 1.
\end{align*} 
Here we used \eqref{reflection} and the assumption $(\a + \a\g,  \b)=(\g, \e) = 0$ in the last equality. 

If $\m = \ob\g$, by \eqref{cancellation},  \eqref{shift} and \eqref{geom}, 
we can show that both sides are equal to 
\begin{align*}
\oa\og\b (1-x) + \frac{g(\a\ob)}{g(\a\g\ob)g(\og)} \og (x).
\end{align*}
This proves the lemma.
\end{proof}

By applying Theorem \ref{first trans} to specific maps $f$,  we have the following reduction formulas, 
which are finite field analogues of \eqref{cor2.3} and \eqref{cor2.4}.

\begin{cor}\label{firstcor}
Assume $(\a + \a\g,  \b) = (\g, \e) = 0$.
\begin{enumerate}
\item  For any $x, y \in \kappa^{*}$ with $x \neq 1$,  we have
\begin{align*}
F_{1: 0:0}^{1:1:1} \left( \left. 
\begin{matrix}
\a \\
\b
\end{matrix}
;
\begin{matrix}
\g  \\
\hspace{1mm}
\end{matrix}
;
\begin{matrix}
\b\og \\
\hspace{1mm}
\end{matrix} 
\right| 
x,  y
\right)  &= \oa (1-x) \gh{\a,  \b\og}{\b}{\frac{x-y}{x-1}} \\
& \hspace{5mm} + \oa (1-x) \delta (x-y)   
+ \frac{\pc{\e}{\b}}{ \pc{\oa}{\b} } \og (x) \ob \g (y).
\end{align*}

\item Assume $(\varphi,  \e+\rho) =0$. For any $x \in \kappa \backslash \{ 0, 1 \}$, we have
\begin{align*}
F_{1: 0:1}^{1:1:2} \left( \left. 
\begin{matrix}
\a \\
\b
\end{matrix}
;
\begin{matrix}
\g  \\
\hspace{1mm}
\end{matrix}
;
\begin{matrix}
\b\og , \varphi \\
\rho
\end{matrix} 
\right| 
x,  x
\right) 
&= \oa (1-x) 
{}_{3}F_{2} \left( 
\begin{matrix}
\a,  \b\og,  \overline{\varphi} \rho  \\
\b , \rho
\end{matrix}
;
\frac{x}{x-1}
\right) \\
&\hspace{20mm}+ \frac{\p{\varphi}{\ob\g} \pc{\e}{\b} }{\pc{\rho}{\ob\g}  \pc{\oa}{\b} } \ob (x).
\end{align*}
\end{enumerate}
\end{cor}
\begin{proof}
(i) By applying Theorem \ref{first trans} to $f(\mu) \equiv 1$,  we have 
\begin{align*}
& F_{1: 0:0}^{1:1:1} \left( \left. 
\begin{matrix}
\a \\
\b
\end{matrix}
;
\begin{matrix}
\g  \\
\hspace{1mm}
\end{matrix}
;
\begin{matrix}
\b\og \\
\hspace{1mm}
\end{matrix} 
\right| 
x,  y
\right) 
= \frac{1}{(1-q)^{2}} 
\sum_{\n,  \m \in \widehat{\kappa^{*}}} \frac{ \p{\a}{\n\m} \p{\g}{\n} \p{\b\og}{\m}   }{  \pc{\b}{\n\m}  \pc{\e}{\n} \pc{\e}{\m}  } \n (x) \m (y)  \\
&= \frac{1}{(1-q)^{2}} 
\oa (1-x) 
\sum_{\eta \in \widehat{\kappa^{*}}} \frac{\p{\a + \b\og }{\eta} }{\pc{\b + \e}{\eta}  } \eta \left( \frac{x}{x-1} \right) 
\sum_{\m \in \widehat{\kappa^{*}}} \frac{ \p{\oeta}{\m}  }{\pc{\e}{\m}} \m \left( \frac{y}{x} \right)  \\
&\hspace{40mm}+\og(x)\ob\g(y) \frac{\pc{\e}{\b}}{\pc{\oa}{\b}} \\
&= \frac{1}{1-q} \oa (1-x) 
\sum_{\eta \in \widehat{\kappa^{*}}} \frac{\p{\a + \b\og }{\eta} }{\pc{\b + \e}{\eta}  } \eta \left( \frac{x}{x-1} \right) 
{}_{1}F_{0} \left(
\begin{matrix}
\oeta \\
\hspace{1mm}
\end{matrix}
;
\frac{y}{x} 
\right) \\
&\hspace{40mm}+ \og(x)\ob\g(y) \frac{\pc{\e}{\b}}{\pc{\oa}{\b}} .
\end{align*}
If we use \eqref{geom},   the sum above coincides with 
\begin{align*}
&\sum_{\eta \neq \e} \frac{\p{\a + \b\og }{\eta} }{\pc{\b + \e}{\eta}  }  \eta \left( \frac{x}{x-1} \right) \eta \left( 1 - \frac{y}{x} \right) 
-q \delta \left( 1 - \frac{y}{x}  \right) +1 \\
&= \sum_{\eta \in \widehat{\kappa^{*}}} \frac{\p{\a + \b\og }{\eta} }{\pc{\b + \e}{\eta}  } \eta \left( \frac{x-y}{x-1} \right) 
- \e \left( \frac{x-y}{x-1} \right) -q \delta \left( x-y \right) +1\\
&= (1-q) \gh{\a,  \b\og}{\b}{\frac{x-y}{x-1}} + (1-q)\delta (x-y) .
\end{align*}
Therefore we obtain the formula.

(ii) If we apply Theorem \ref{first trans} to $\displaystyle f(\m) = \p{\varphi}{\m} /  \pc{\rho}{\m}$ and $x=y$,  we have
\begin{align*}
& F_{1: 0:1}^{1:1:2} \left( \left. 
\begin{matrix}
\a \\
\b
\end{matrix}
;
\begin{matrix}
\g  \\
\hspace{1mm}
\end{matrix}
;
\begin{matrix}
\b\og , \varphi \\
\rho
\end{matrix} 
\right| 
x,  x
\right) \\
&= \frac{1}{(1-q)^{2}}
\sum_{\n,  \m \in \widehat{\kappa^{*}}} \frac{ \p{\a}{\n\m} \p{\g}{\n}  \p{\b\og }{\m} \p{\varphi}{\m}  }{  \pc{\b}{\n\m} \pc{\rho}{\m} \pc{\e}{\n} \pc{\e}{\m}  } \n (x) \m (x)  \\
&= \frac{1}{(1-q)^{2}} \oa (1-x) 
\sum_{\eta \in \widehat{\kappa^{*}}} \frac{\p{\a + \b\og }{\eta}  }{\pc{\b + \e }{\eta} } \eta \left( \frac{x}{x-1} \right) 
\sum_{\m \in \widehat{\kappa^{*}}} \frac{\p{  \varphi + \oeta }{\m}  }{\pc{ \rho + \e }{\m} } \m (1)   \\
&\hspace{40mm}+ \frac{\p{\varphi}{\ob\g}  \pc{\e}{\b} }{\pc{\rho}{\ob\g} \pc{\oa}{\b} } \ob (x) .
\end{align*}
Note that we assume $(\varphi , \rho +  \e) =0$. By \eqref{Vande},  we have
\begin{align*}
\frac{1}{1-q}\sum_{\m \in \widehat{\kappa^{*}}} \frac{\p{  \varphi + \oeta }{\m}  }{\pc{ \rho + \e }{\m} }  \m (1) 
= \gh{\varphi,  \oeta}{\rho}{1}
= \frac{\p{ \overline{\varphi} \rho }{\eta}  }{\pc{\rho}{\eta}} .
\end{align*}
Therefore,  we obtain 
\begin{align*}
& \frac{1}{(1-q)^{2}}
\sum_{\eta \in \widehat{\kappa^{*}}} \frac{\p{\a + \b\og }{\eta}  }{\pc{\b + \e }{\eta} } \eta \left( \frac{x}{x-1} \right) 
\sum_{\m \in \widehat{\kappa^{*}}} \frac{\p{  \varphi + \oeta }{\m}  }{\pc{ \rho + \e }{\m} } \m (1)    \\
&= \frac{1}{1-q} \sum_{\eta \in \widehat{\kappa^{*}}} \frac{\p{\a  + \b\og + \overline{\varphi} \rho}{\eta}  }{\pc{\b + \rho + \e}{\eta}  } \eta \left( \frac{x}{x-1} \right)  
=   {}_{3}F_{2} \left( 
\begin{matrix}
\a,  \b\og,  \overline{\varphi}\rho \\
\b,  \rho
\end{matrix}
;
\frac{x}{x-1}
\right).
\end{align*}
This proves the formula.
\end{proof}

\begin{remk}
The formula (i) is originally obtained by Li-Li-Mao \cite[Corollary 3.1]{llm} and  also follows from \cite[Theorem 1.2]{tsb} and \cite[Theorem 3.4]{tripathi}.
\end{remk}

%%%%%%%%%%%%%%%%%%%%%%%%%%% Ito's part1 %%%%%%%%%%%%%%%%%%%%%%%%%%%%%%%%%%%

%%%%%%%%%%%%%%%%%%%%%%%%%%% Ito's part2 %%%%%%%%%%%%%%%%%%%%%%%%%%%%%%%%%%%

Next,  we give an analogous formula of the following transformation formula \cite[Theorem 2.3]{liuwang}:
\begin{align*}
\begin{split}
&\sum_{i,j=0}^{\infty} \Omega(j) \frac{\p{a}{i+j} \p{c}{i} \p{b - c}{j} }{\p{b}{i+j}} \frac{x^{i} y^{j}}{i!j!} \\
&= (1-x)^{b-a-c} \sum_{n=0}^{\infty} \frac{\p{b-a}{n} \p{b - c}{n} }{\p{b}{n} n! } x^{n}
\sum_{j=0}^{n} \Omega (j) \frac{\p{-n}{j} \p{a}{j}  }{j! \p{1+a-b-n}{j}} \left( \frac{y}{x} \right)^{j},
\end{split}
\end{align*}
where $\Omega(j)$ is a complex sequence.

\begin{theo}\label{second trans}
Assume $(\a + \a\g,  \b) = (\g,  \e) =0$.  Let $f\colon\widehat{\kappa^{*}} \rightarrow \mathbb{C}$ be a map.  
For any $x, y \in \kappa$ with $x \neq 0,  1$,  we have
\begin{align*}
&\sum_{\n ,  \m \in \widehat{\kappa^{*}}} f(\m) 
\frac{\p{\a}{\n\m} \p{\g}{\n} \p{\b\og}{\m} }{\pc{\b}{\n\m} \pc{\e}{\n} \pc{\e}{\m} } \n (x) \m (y) \\
&=\oa\og\b (1-x) \sum_{\eta \in \widehat{\kappa^{*}}} \frac{\p{\b\oa + \b\og}{\eta}  }{\pc{\b + \e }{\eta}} \eta (x) 
\sum_{\m \in \widehat{\kappa^{*}}} f(\m) \frac{\p{\a + \oeta}{\m} }{\pc{\a\ol{\b\eta} + \e}{\m}} \m \left(\frac{y}{x} \right)  \\
& - (1-q)^{2}  f(\oa)  \a\ob (-x) \oa (-y) \oa\og\b (1-x) \frac{\pc{\e}{\b}}{\p{\og}{\b}}  
+(1-q)^{2} f(\g\ob)   \og (x) \g\ob (y)  \frac{\pc{\e}{\b}}{\pc{\oa}{\b}}.
\end{align*}
\end{theo}
\begin{proof}
By \eqref{transitivity}, we have  
\begin{align*}
&\sum_{\n, \m \in \widehat{\kappa^{*}} } f(\m) 
\frac{\p{\a}{\n\m} \p{\g}{\n} \p{\b\og}{\m} }{\pc{\b}{\n\m} \pc{\e}{\n} \pc{\e}{\m} } \n (x) \m (y)   \\
&=\sum_{\m \in \widehat{\kappa^{*}} } f(\m ) \frac{\p{\a + \b\og}{\m} }{\pc{\b + \e}{\m} }  \m (y)
\sum_{\n \in \widehat{\kappa^{*}} } \frac{\p{\a\m + \g }{\n}  }{\pc{\b\m + \e }{\n} } \n (x)   \\
&= (1-q) \sum_{\m \in \widehat{\kappa^{*}}} f(\m ) \frac{\p{\a + \b\og}{\m} }{\pc{\b + \e}{\m} } \m (y)
\gh{\a\m,  \g}{\b\m}{x}.   
\end{align*}
By Lemma \ref{euler1} below,   we obtain
\begin{align*}
& (1-q) \sum_{\m \in \widehat{\kappa^{*}}} f(\m ) \frac{\p{\a + \b\og}{\m} }{\pc{\b + \e}{\m} } \m (y)
\gh{\a\m,  \g}{\b\m}{x}\\
&= (1-q)\sum_{\m \in \widehat{\kappa^{*}}} f(\m )  \frac{\p{\a + \b\og}{\m} }{\pc{\b + \e}{\m} } \m (y)
\oa\og\b (1-x) \gh{\oa\b,  \og\b\m}{\b\m}{x} \\
& \hspace{15mm} -q^{-1} (1-q)^{2} f(\oa)\frac{\p{\a + \b\og }{\oa}  }{ \pc{\b + \e}{\oa}} 
\oa (y) \frac{\p{\a\ob}{\g}}{\pc{\e}{\g}} \a\ob (-x) \oa\og\b (1-x)   \\
& \hspace{20mm} +(1-q)^{2} 
f(\ob\g) 
\frac{ \p{\a + \b\og }{\ob\g}   }{ \pc{\b + \e }{\ob\g} }  \ob\g (y) 
\frac{ \pc{\e}{\g} }{ \p{\a\ob}{\g} } \og (-x)  \\
&=  \oa\og\b (1-x) 
\sum_{\n,  \m \in \widehat{\kappa^{*}}} f(\m ) \frac{  \p{\b\og}{\n\m}  \p{\oa\b}{\n}  \p{\a}{\m}   }{\pc{\b}{\n\m}  \pc{\e}{\m} \pc{\e}{\n} } 
\n (x) \m (y)  \\
& \hspace{15mm} - (1-q)^{2} f(\oa) 
\oa (-y) \a\ob (-x) \oa\og\b (1-x)  \frac{\pc{\e}{\b}}{\p{\og}{\b}} \\
& \hspace{20mm} +(1-q)^{2} 
f(\ob\g) \og (x) \ob\g (y) \frac{\pc{\e}{\b}}{\pc{\oa}{\b}}.
\end{align*}
Here we used \eqref{reflection} in the last equality to simplify the second and third terms.  
Finally,  if we put $\n\m = \eta$,  the first term agrees with 
\begin{align*}
& \oa\og\b (1-x) 
\sum_{\eta \in \widehat{\kappa^{*}}} \frac{ \p{\b\og + \b\oa }{\eta}  }{   \pc{\b + \e }{\eta}  } \eta (x) 
\sum_{\m \in \widehat{\kappa^{*}}} f(\m) \frac{   \p{\oa\b\eta}{\om} \p{\a}{\m} }{    \pc{\eta}{\om} \pc{\e}{\m} } \m \left( \frac{y}{x} \right)  \\
&=  \oa\og\b (1-x) 
\sum_{ \eta \in \widehat{\kappa^{*}}} \frac{  \p{\b\og + \b\oa }{\eta}  }{   \pc{\b + \e }{\eta}  } \eta (x) 
\sum_{ \m \in \widehat{\kappa^{*}}  } f(\m) \frac{  \p{\oeta + \a }{\m}  }{ \pc{\a\ol{\b\eta} + \e }{\m}  } \m \left( \frac{y}{x} \right) .
\end{align*}
This proves the theorem.
\end{proof}

\begin{lemm}\label{euler1}
Assume $(\a + \a\g,  \b) = (\g,  \e) =0$.  For  any $\m \in \widehat{\kappa^{*}}$ and $x \in \kappa \backslash \{ 1 \}$,
we have
\begin{align*}
\gh{\a\m,  \g}{\b\m}{x} 
=& \oa\og\b (1-x) \gh{\oa\b,  \og\b\m}{\b\m}{x}  \\
& \hspace{20mm}  +  \delta (\a \m) (1-q^{-1}) \frac{\p{\a\ob}{\g}}{\pc{\e}{\g}} \a\ob(-x) \oa\og\b (1-x)   \\
& \hspace{20mm} + \delta (\og\b\m) (1-q) \frac{\pc{\e}{\g}}{\p{\a\ob}{\g}} \og (-x).
\end{align*}
\end{lemm}
\begin{proof}
The case when $x=0$ is trivial, hence we may assume $x\neq 0$.
The formula in the case $(\a\m + \g,  \e + \b\m) =0$, that is, $\m \neq  \oa,  \ob\g$ 
follows from \eqref{euler0}. 
If $\m = \oa$ (resp.  $\m = \ob\g$),  one can show that both sides equal
\begin{align*}
&\frac{g(\a\g\ob) g(\oa\b) }{g(\g)}  \a\ob (x) \oa\og\b (1-x) +1 \\
&\left( resp. \hspace{3mm}
\oa\og\b (1-x) + \frac{g(\a\ob)}{g(\a\g\ob) g(\og) }  \og (x) 
\right)
\end{align*}
by similar computation to the proof of Lemma \ref{pfaff1}.
\end{proof}

We have the following analogous formulas of  the reduction formulas \eqref{cor2.10}, \eqref{cor2.11} and \eqref{cor2.13}.  We remark that the function ${}_{3}F_{2} (x)$ in \eqref{cor2.13} reduces to ${}_{2}F_{1}(x)$ in the case of finite fields since the complex parameters $1+b/2$ and $b/2$ agree with the same character.

\begin{cor}\label{secondcor}
\begin{enumerate}
\item Assume $(\a + \a\g,  \b) = (\b + \g,  \e) =0$. For any $x \in \kappa \backslash \{ 0,  1 \}$, we have
\begin{align*}
F_{1:0:0}^{1:1:1} \left( \left.
\begin{matrix}
\a \\
\b
\end{matrix}
; 
\begin{matrix}
\g \\
\hspace{1mm}
\end{matrix}
; 
\begin{matrix}
\b\og \\
\hspace{1mm}
\end{matrix}
\right|
x,x
\right)
= \oa (1-x) + \ob (x) \frac{\pc{\e}{\b}}{\pc{\oa}{\b}}.
\end{align*}

\item Assume $(\a + \a\g,  \b) = (\b + \g,  \e) = (\a + \varphi,  \e + \a\ob) =0$. For any $x \in \kappa \backslash \{ 0,  1 \}$, we have
\begin{align*}
F_{1:0:1}^{1:1:2} \left( \left.
\begin{matrix}
\a \\
\b
\end{matrix}
; 
\begin{matrix}
\g \\
\hspace{1mm}
\end{matrix}
; 
\begin{matrix}
\b\og ,  \varphi \\
\b\varphi
\end{matrix}
\right|
x,x
\right)  
& = \oa\og\b (1-x) \gh{\b\og,  \b\varphi \oa}{\b\varphi}{x} \\
&+ \ob (-x)  \oa\b (1-x) \frac{ \pc{\varphi}{\b} \pc{ \e}{\b} }{\pc{\oa}{\b} \p{ \og}{\b}  }
+\ob (-x) q^{-\delta (\varphi\g\ob)} \frac{\pc{\varphi}{\b} \pc{\e}{\b} }{\pc{\oa}{\b}\p{\ol{\varphi\g}}{\b} }.
\end{align*}

\item Assume $(\a^{2} + \a^{2}\g, \b^{2}) = (\a^{2} + \b^{2}+ \g, \e) =0$. For any $x \in \kappa \backslash \{ 0, 1 \}$, we have  
\begin{align*}
F_{1:0:1}^{1:1:2} \left( \left.
\begin{matrix}
\a^{2} \\
\b^{2}
\end{matrix}
; 
\begin{matrix}
\g \\
\hspace{1mm}
\end{matrix}
; 
\begin{matrix}
\b^{2}\og ,  \a\ob \\
\a\b
\end{matrix}
\right|
x,x
\right)
&= \ol{\a^{2} \g} \b^{2} (1-x)   {}_{2}F_{1} \left( 
\begin{matrix}
\b^{2} \og,  \b\oa \\
\a\b
\end{matrix}
; x
\right)  \\
&+ \ob^{2} (x) \oa^{2}\b^{2}(1-x)\frac{\pc{\a\ob}{\b^{2}} \pc{\e}{\b^{2}} }{\p{\oa^{2}}{\b^{2}} \p{\og}{\b^{2}} }   \\
&+\ob^{2}(x) \frac{\p{\ol{\a\b}}{\b^{2}} \pc{\e}{\b^{2}} }{\pc{\b\ol{\a\g}}{\b^{2}} \pc{\oa^{2}}{\b^{2}} } . 
\end{align*}
\end{enumerate}
\end{cor}

\begin{proof}
(i) If we apply Theorem \ref{second trans} to $f (\mu) \equiv 1$ and put $x = y$,  we have 
\begin{align*}
&F_{1:0:0}^{1:1:1} \left( \left.
\begin{matrix}
\a \\
\b
\end{matrix}
; 
\begin{matrix}
\g \\
\hspace{1mm}
\end{matrix}
; 
\begin{matrix}
\b\og \\
\hspace{1mm}
\end{matrix}
\right|
x,x
\right) \\
&= \frac{1}{(1-q)^{2}} 
\sum_{\n,  \m \in \widehat{\kappa^{*}}} \frac{\p{\a}{\n\m} \p{\g}{\n}  \p{\b\og}{\m}  }{\pc{\b}{\n\m} \pc{\e}{\n} \pc{\e}{\m} } \n (x) \m (x)  \\
&= \frac{1}{(1-q)^{2}}  \oa\og\b (1-x) \sum_{\eta \in \widehat{\kappa^{*}}} \frac{\p{\b\og + \b\oa}{\eta}  }{\pc{\b + \e }{\eta} } 
\eta (x) \sum_{\m \in \widehat{\kappa^{*}}} \frac{\p{\oeta + \a}{\m}  }{\pc{\a\ol{\b\eta} + \e }{\m} } \m (1) \\
&  \hspace{20mm}-  \ob (-x)  \oa\og\b (1-x)  \frac{\pc{\e}{\b}}{\p{\og}{\b}}  
 + \ob (x)  \frac{\pc{\e}{\b}}{\pc{\oa}{\b}} .
\end{align*}
The first term equals
\begin{align}
\frac{1}{1-q} \oa\og\b (1-x) \sum_{\eta \in \widehat{\kappa^{*}}} \frac{\p{\b\og + \b\oa}{\eta}  }{\pc{\b + \e }{\eta}  } \eta (x) 
\gh{\a,  \oeta}{\a\ol{\b\eta}}{1}.  \label{second2}
\end{align}
Since we assume $\b \neq \e$,  we have,  by \eqref{eulergauss} and \eqref{prelection},  
\begin{align*}
\gh{\a,  \oeta}{\a\ol{\b\eta}}{1}
= \frac{\gc{\a\ol{\b\eta}} g(\ob) }{\gc{\ol{\b\eta}} \gc{\a\ob} } 
= \frac{\pc{\a\ob}{\oeta}}{\pc{\ob}{\oeta}} = \frac{\p{\b}{\eta}}{\p{\oa\b}{\eta}}.
\end{align*}
Therefore,  if we use  \eqref{cancellation} and  \eqref{geom}, we can see that \eqref{second2} equals
\begin{align*}
\oa\og\b (1-x) \gh{\b\og,  \b }{\b}{x} 
&= \oa\og\b (1-x) \left(    {}_{1}F_{0} \left( 
\begin{matrix}
\b\og \\
\hspace{1mm}
\end{matrix}
; x 
\right) + q^{-1} \frac{\p{\b\og}{\ob}}{\pc{\e}{\ob}} \ob (x)  \right) \\
&= \oa (1-x) + \oa\og\b (1-x) \frac{  \pc{\e}{\b}  }{  \p{\og}{\b}  } \ob (-x).
\end{align*}
Here we used $(\b,  \e + \g) =0$.   This proves the formula.

(ii) By applying Theorem \ref{second trans} to $f(\m) = \p{\varphi}{\m} / \pc{\b\varphi}{\m}$ and putting $x=y$,  we have
\begin{align*}
&F_{1:0:1}^{1:1:2} \left( \left.
\begin{matrix}
\a \\
\b
\end{matrix}
; 
\begin{matrix}
\g \\
\hspace{1mm}
\end{matrix}
; 
\begin{matrix}
\b\og ,  \varphi \\
\b\varphi
\end{matrix}
\right|
x,x
\right) \\
&= \frac{1}{1-q} 
\oa\og\b (1-x) \sum_{\eta \in \widehat{\kappa^{*}}} \frac{\p{\b\og + \b\oa}{\eta} }{\pc{\b + \e }{\eta} } \eta (x) 
{}_{3}F_{2} \left( 
\begin{matrix}
\varphi,  \a,  \oeta \\
\b\varphi,  \a\ol{\b\eta}
\end{matrix}
; 1
\right) \\
& \hspace{20mm}  - \ob(-x) \oa\og\b (1-x) \frac{ \p{\varphi}{\oa} \pc{\e}{\b}}{ \pc{\b\varphi}{\oa} \p{\og}{\b}}  
+ \ob(x)  \frac{ \p{\varphi}{\ob\g}  \pc{\e}{\b}}{ \pc{\b\varphi}{\ob\g}  \pc{\oa}{\b}} .
\end{align*}
By \eqref{reflection},  the third term coincides with 
\begin{align*}
\ob (-x) q^{-\delta (\varphi\g\ob)} \frac{\pc{\varphi}{\b} \pc{\e}{\b} }{\pc{\oa}{\b}  \p{\ol{\varphi\g}}{\b} }.
\end{align*}
Note that  the above ${}_{3}F_{2}(1)$ in the first term is Saalsch\"utzian and $\varphi + \a + \oeta \neq  \e + \b\varphi + \a\ol{\b\eta} $
under our assumptions.  
Hence,  by  \eqref{saalschutz},  the value ${}_{3}F_{2} (1)$ is equal to  
\begin{align*}
\frac{\p{\b}{\eta}  \p{\b\varphi\oa}{\eta}  }{\pc{\b\varphi}{\eta} \p{\oa\b}{\eta} }
 + \frac{\gc{\b\varphi} \gc{\a\ob} }{g (\varphi) g(\a)} \frac{\pc{\e}{\eta}}{\p{\oa\b}{\eta}}. 
\end{align*}
Therefore we obtain
\begin{align*}
&\frac{1}{1-q} 
\oa\og\b (1-x) \sum_{\eta \in \widehat{\kappa^{*}}} \frac{\p{\b\og}{\eta} \p{\b\oa}{\eta} }{\pc{\b}{\eta} \pc{\e}{\eta} } \eta (x) 
{}_{3}F_{2} \left( 
\begin{matrix}
\varphi,  \a,  \oeta \\
\b\varphi,  \a\ob\oeta
\end{matrix}
; 1
\right) \\
&= \oa\og\b (1-x) {}_{3}F_{2} \left( 
\begin{matrix}
\b\og,    \b,  \b\varphi\oa  \\
\b,  \b\varphi
\end{matrix}
; x 
\right) 
+ \oa\og\b (1-x)   \frac{\gc{\b\varphi} \gc{\a\ob} }{g (\varphi) g(\a)} 
F\left( \b\og,  \b;  x \right) .
\end{align*}
By \eqref{cancellation},  the first term equals
\begin{align*}
\oa\og\b (1-x) \gh{\b\og,  \b\varphi \oa}{\b\varphi}{x} 
+  \oa \og \b (1-x) q^{-1} \frac{\p{\b\og + \b\varphi\oa}{\ob}  }{\pc{\b\varphi + \e }{\ob} } \ob (x).
\end{align*}
The formulas \eqref{shift} and \eqref{geom} imply that  the second term is equal to
\begin{align*}
&\oa\og\b (1-x)   \frac{\gc{\b\varphi} \gc{\a\ob} }{g (\varphi) g(\a)} 
\frac{\p{\b\og}{\ob}}{\pc{\b}{\ob}} \ob (x)  {}_{1}F_{0} 
\left( 
\begin{matrix}
\og \\
\hspace{1mm}
\end{matrix}
; 
x
\right) \\
&=  \ob (-x)  \oa\b (1-x) \frac{ \pc{\varphi}{\b} \pc{ \e}{\b} }{\pc{\oa}{\b} \p{ \og}{\b}  }.
\end{align*}
Finally,  it is easily verified the equality
\begin{align*}
\oa \og \b (1-x) q^{-1} \frac{\p{\b\og + \b\varphi\oa}{\ob}  }{\pc{\b\varphi + \e }{\ob} } \ob (x) 
&= \ob(-x) \oa\og\b (1-x) \frac{ \p{\varphi}{\oa} \pc{\e}{\b}}{ \pc{\b\varphi}{\oa} \p{\og}{\b}} , 
\end{align*}
hence we obtain the formula.

(iii) If we apply Theorem \ref{second trans} to $f(\m) = \p{\a\ob}{\m} / \pc{\a\b}{\m}$ and put $x=y$, we have 
\begin{align*}
F_{1:0:1}^{1:1:2} \left( \left.
\begin{matrix}
\a^{2} \\
\b^{2}
\end{matrix}
; 
\begin{matrix}
\g \\
\hspace{1mm}
\end{matrix}
; 
\begin{matrix}
\b^{2}\og ,  \a\ob \\
\a\b
\end{matrix}
\right|
x,x
\right)
&=\frac{1}{1-q} \ol{\a^{2}\g} \b^{2} (1-x) \sum_{\eta \in \widehat{\kappa^{*}}} \frac{\p{\b^{2}\oa^{2} + \b^{2}\og }{\eta} }{\pc{\b^{2} +\e}{\eta}} \eta (x) {}_{3}F_{2} \left( 
\begin{matrix}
\a^{2}, \a\ob,  \oeta \\
\a\b, \a^{2} \ol{\b^{2}\eta}
\end{matrix}
; 1
\right)  \\
&- \ob^{2} (x) \ol{\a^{2}\g}\b^{2}(1-x)  \frac{\p{\a\ob}{\oa^{2}} \pc{\e}{\b^{2}} }{\pc{\a\b}{\oa^{2}} \p{\og}{\b^{2}} }  \\
&+ \ob^{2}(x) \frac{\p{\a\ob}{\g\ob^{2}} \pc{\e}{\b^{2}} }{\pc{\a\b}{\g\ob^{2}} \pc{\oa^{2}}{\b^{2}}}.
\end{align*}
By \eqref{reflection}, the second term agrees with 
\begin{align*}
- \ob^{2} (x) \ol{\a^{2}\g}\b^{2}(1-x)  \frac{ \pc{\e}{\b^{2}} }{\p{\og}{\b^{2}} }. 
\end{align*}
Note that the above ${}_{3}F_{2}(1)$ is Saalsch\"utzian and we have $\a^{2} + \a\ob + \oeta \neq  \e + \a\b + \a^{2}\ol{\b^{2}\eta}$ under our assumption, hence, the value ${}_{3}F_{2}(1)$ is equal to 
\begin{align*}
\frac{\p{\b^{2}}{\eta} \p{\oa\b}{\eta} }{\pc{\a\b}{\eta} \p{\oa^{2}\b^{2}}{\eta}} + \frac{\gc{\a\b} \gc{\a^{2}\ob^{2} }}{g(\a^{2}) g(\a\ob)} \frac{\pc{\e}{\eta}}{\p{\oa^{2}\b^{2}}{\eta}}
\end{align*}
by \eqref{saalschutz}. Therefore we obtain 
\begin{align*}
&\frac{1}{1-q} \ol{\a^{2}\g} \b^{2} (1-x) \sum_{\eta \in \widehat{\kappa^{*}}} \frac{\p{\b^{2}\oa^{2} + \b^{2}\og }{\eta} }{\pc{\b^{2} +\e}{\eta}} \eta (x) {}_{3}F_{2} \left( 
\begin{matrix}
\a^{2}, \a\ob,  \oeta \\
\a\b, \a^{2} \ol{\b^{2}\eta}
\end{matrix}
; 1
\right) \\
&= \ol{\a^{2}\g} \b^{2} (1-x) {}_{3}F_{2} \left( 
\begin{matrix}
\b^{2}\og, \b^{2}, \b\oa \\
\b^{2}, \a\b
\end{matrix}
; x
\right) 
+ \ol{\a^{2}\g} \b^{2} (1-x) \frac{\gc{\a\b} \gc{\a^{2}\ob^{2} }}{g(\a^{2}) g(\a\ob)} F\left( \b^{2}\og, \b^{2}; x \right) \\
&=  \ol{\a^{2}\g} \b^{2} (1-x) {}_{2}F_{1} \left( 
\begin{matrix}
\b^{2}\og,  \b\oa \\
\a\b
\end{matrix}
; x
\right)  +  \ol{\a^{2}\g} \b^{2} (1-x) q^{-1} \frac{\p{\b^{2}\og + \b\oa}{\ob^{2}}}{\pc{\e + \a\b}{\ob^{2}}} \ob^{2}(x) \\
&+ \ol{\a^{2}\g} \b^{2} (1-x) \frac{\gc{\a\b} \gc{\a^{2}\ob^{2} }}{g(\a^{2}) g(\a\ob)} F\left( \b^{2}\og, \b^{2}; x \right). \\
\end{align*}
Here we used \eqref{cancellation} in the last equality. 
If we use \eqref{reflection} (resp. \eqref{shift} and \eqref{geom}), one can show that the second (resp. third) term agrees with 
\begin{align*}
& \ob^{2}(x) \ol{\a^{2}\g} \b^{2} (1-x)  \frac{\pc{\e}{\b^{2}}}{\p{\og}{\b^{2}}}. \\
&\left( resp. \hspace{1mm}\ob^{2}(x) \oa^{2}\b^{2}(1-x) \frac{\pc{\a\ob}{\b^{2}} \pc{\e}{\b^{2}} }{\p{\oa^{2}}{\b^{2}} \p{\og}{\b^{2}}}. \right)
\end{align*}
Therefore we obtain the formula.
\end{proof}

\begin{remk}
Corollary \ref{secondcor} (i) can also be derived from  Corollary \ref{firstcor} (i). 
\end{remk}

By specializing parameters and arguments in the formulas above, we have the following summation formulas, which are finite field analogues of \cite[Corollary 5.1]{liuwang}.

\begin{cor} 
Assume $p \neq 2$.
\begin{enumerate}
\item Suppose that $(\a^{2} + \a^{2}\g^{2}, \b^{2}) = (\a^2+\b^2 + \g^{2}+ \phi\a\overline{\b\g},\e)=0$. Then, we have
\begin{align*}
&\kF{1:1:2}{1:0:1}{\a^2}{\b^2}{\g^2}{}{\b^2\og^2,\phi\a\overline{\b\g}}{\phi\a\b\og}{-1,-1}
=\overline{\a\g}\b(4)\sum_{\chi^2=\e}\dfrac{\gc{\phi\a\b\og}g(\b\og\chi)}{g(\b^2\og^2)g(\phi\a\chi)}\\
&+\oa\b(4)\dfrac{\pc{\phi\a\overline{\b\g}}{\b^2} \pc{\e}{\b^{2}} }{\pc{\oa^2}{\b^2}\p{\og^2}{\b^2}}+q^{-\delta(\phi\a\ob^3\g)}\dfrac{\pc{\phi\a\overline{\b\g}}{\b^{2}}\pc{\e}{\b^2}}{\pc{\oa^2}{\b^2}\p{\phi\ol{\a\g}\b}{\b^2}}.
\end{align*}

\item Suppose that $(\a +\a\g^{2}, \b^{2}) = (\a+\b^2+\g^{2} + \b^2\og^2,\e)=(\overline{\a\g^2},\e+\a\ob^2)=0$. Then,
we have
\begin{align*}
&\kF{1:1:2}{1:0:1}{\a}{\b^2}{\g^2}{}{\b^2\og^2,\ol{\a\g^2}}{\ol{\a\g^{2}}\b^2}{\dfrac{1}{2},\dfrac{1}{2}}
=\a\ob^2\g^2(2)\sum_{\chi^2=\e} \dfrac{\gc{\b^2 \ol{\a\g^{2}}} g(\phi)}{g(\phi\b\og\chi) g(\ol{\a\g}\b\chi)} \\
&+\a(2)\frac{ \pc{\ol{\a\g^2}}{\b^{2}} \pc{\e}{\b^{2}} }{ \pc{\oa}{\b^{2}} \p{\og^2}{\b^2} }+\b(4)q^{-\delta(\a\b^2)}\dfrac{ \pc{\ol{\a\g^2}}{\b^{2}} \pc{\e}{\b^2} }{ \pc{\oa}{\b^{2}} \p{\a}{\b^2} }.
\end{align*}

\item Suppose that $(\a^{2} + \a^{2}\g^{2}, \b^{2}) = (\a^2+\b^2+ \g^{2} + \b^{2} \og^{2} + \a^{2}\g^{2} \ob^{4}+ \a^2\g^4\ob^4,\e)=0$. Then, we have
\begin{align*}
&\kF{1:1:2}{1:0:1}{\a^2}{\b^2}{\g^2}{}{\b^2\og^2,\a^2\g^2\ob^4}{\a^2\g^2\ob^2}{\dfrac{1}{2},\dfrac{1}{2}}
=\a\ob\g(4)\sum_{\chi^2=\e}\dfrac{\gc{\a\g\ob}\gc{\phi\a\g\ob}}{g(\phi\a\chi) g(\a\g^2\ob^2\chi)}\\
&+\a(4)\dfrac{\pc{\a^2\g^2\ob^4}{\b^{2}}\pc{\e}{\b^2}}{\pc{\oa^2}{\b^2}\p{\og^2}{\b^2}}
+\b(4)q^{-\delta(\a^2\g^4\ob^6)} \dfrac{\pc{\a^2\g^2\ob^4}{\b^{2}}\pc{\e}{\b^2}}{\pc{\oa^2}{\b^{2}}\p{\ol{\a^{2}\g^{4}}\b^4}{\b^2}}.
\end{align*}
\end{enumerate}
\end{cor}

\begin{proof}
Part (i) (resp. (ii), (iii)) follows from Corollary \ref{secondcor} (ii) and \eqref{kummer} (resp. \eqref{gsecond}, \eqref{bailey}).

%\begin{enumerate}
%\item Using Corollary \ref{secondcor} (ii) and \eqref{kummer}, we obtain (i).
%
%\item Using Corollary \ref{secondcor} (ii) and \eqref{gsecond}, we obtain (ii).
%
%\item Using Corollary \ref{secondcor} (ii) and \eqref{bailey}, we complete the proof.
%\end{enumerate}
\end{proof}

%%%%%%%%%%%%%%%%%%%%%%%%%%% Ito's part2 %%%%%%%%%%%%%%%%%%%%%%%%%%%%%%%%%%%

\subsection{Formulas for $F_{1:0:C^{\prime}}^{0:2:C}$}

%%%%%%%%%%%%%%%%%%%%%%%%%%% Nemoto's part %%%%%%%%%%%%%%%%%%%%%%%%%%%%%%%%

First, we give a finite field analogue of \cite[Theorem .3.1]{liuwang}:
\begin{align*}
\begin{split}
&\sum_{i,j=0}^{\infty} \Omega(j) \frac{\p{a}{i} \p{c}{i} \p{b - a}{j}\p{b - c}{j} }{\p{b}{i+j}} \frac{x^{i} y^{j}}{i!j!} \\
&= (1-x)^{b-a-c} \sum_{n=0}^{\infty} \frac{\p{b-a}{n} \p{b - c}{n} }{\p{b}{n} n! } x^{n}
\sum_{j=0}^{n} \Omega (j) \frac{\p{-n}{j}  }{j! } \left( \frac{y(x-1)}{x} \right)^{j},
\end{split}
\end{align*}
where $\Omega (j)$ is a complex sequence.

%\begin{align*}
%{}_{2}F_{1} \left( 
%\begin{matrix}
%\a, \g \\
%\a
%\end{matrix}
%;
%x 
%\right) 
%&={}_{1}F_{0} \left( 
%\begin{matrix}
%\g \\
%-
%\end{matrix}
%;
%x 
%\right) 
%+ q^{-1} \cfrac{(\g)_{\oa}}{(\e)^{\circ}_{\oa}} \oa(x) \\
%&= \og(1-x)+\cfrac1q \cdot \cfrac{g^{\circ}(\e)}{g^{\circ}\oa)} \cdot \cfrac{g(\g\oa)}{g(\g)} \oa(x) \\
%&= \og(1-x)+ \cfrac{g(\g\oa)}{g^{\circ}(\oa)g(\g)} \oa(x). 
%\end{align*}

%On the other hand, 
%RHS is 
%\begin{align*}
%&\overline{\a\g}\beta \a\ob(1-x)  {}_{2}F_{1} \left( 
%\begin{matrix}
%\e, \g\a\\
%\a
%\end{matrix}
%;
%x 
%\right) 
%+(1-q)\cfrac{g(\g\oa)}{g(\g)g^{\circ}(\oa)} \oa(x) \\
%&=\g(1-x)(qF(\og \a, \a ; x)+1))+(1-q)\cfrac{g(\g\oa)}{g(\g)g^{\circ}(\oa)} \oa(x).
%\end{align*}

%The former part is transformed as follows. 
%\begin{align*}
%\og(1-x)\left(q\cfrac{(\og \a)_{\oa}}{(\a)^{\circ}_{\oa}} \oa(x) {}_{1}F_{0} \left( 
%\begin{matrix}
%\og \\
%-
%\end{matrix}
%;
%x 
%\right) 
%+1
%\right) 
%&=
%q\cfrac{(\og \a)_{\oa}}{(\a)^{\circ}_{\oa}} \oa(x) + \og(1-x) \\
%&= 
%q \cfrac{g(\g)}{g(\og\a)} \cfrac{g^{\circ}(\a)}{g^{\circ}(\e)} + \og(1-x) \\
%&=\cfrac{\oa(-1)q \cdot \oa(-1)q}{g^{\circ}(\g)g(\oa)} \cdot \cfrac{g^{\circ}(\g\oa)}{\og \a(-1)q} \oa(x)+ \og(1-x) \\
%&=q \cfrac{g^{\circ}(\g \oa)}{g^{\circ}(\g)g(\oa)} \oa(x)+\og(1-x). 
%\end{align*}
%By summing up the latter part, we have the theorem. 
%By same calculation, we have the theorem in case of $\mu= \g\ob$.
%\end{proof}

\begin{theo} \label{thirdtrans}
Assume $(\a+\g+\oa\g, \e)=0$. Let  $f \colon \widehat{\kappa^{*}} \rightarrow \mathbb{C}$ be a map. 
For any $x, y \in \kappa$ with $x \neq 0, 1$,  we have
\begin{align*}
&\sum_{\n, \m \in \widehat{\kappa^{*}}}f(\mu)
\cfrac{ \p{\a + \g}{\n}  \p{ \b\oa+\b\og}{\m} \n(x)\m(y)}{(\b)^{\circ}_{\m\n}(\e)^{\circ}_{\n}(\e)^{\circ}_{\m}} \\
&=
\b\overline{\a\g}(1-x)\sum_{\eta \in \widehat{\kappa^{*}}} \cfrac{\p{\b\oa + \b\og}{\eta}}{\pc{\b + \e}{\eta}} \eta(x)
\sum_{\m \in \widehat{\kappa^{*}}}f(\m) \cfrac{ \p{\oeta}{\m} }{  \pc{\e}{\m} }  \m\left(\cfrac{y(x-1)}{x}\right) \\
&+(1-q)^2 \ob(-y)
\left( 
f(\a\ob) \cfrac{\pc{\e}{\b}}{\p{\og}{\b}} \a\left(-\cfrac{y}{x} \right)
+f(\g\ob) \cfrac{\pc{\e}{\b}}{\p{\oa}{\b}}  \g\left(-\cfrac{y}{x}\right)
\right).
\end{align*}
\end{theo}
\begin{proof}
We have 
\begin{align*}
&\sum_{\n, \m \in \widehat{\kappa^{*}}}f(\mu)
\cfrac{ \p{\a + \g}{\n}  \p{ \b\oa+\b\og}{\m}  \n(x)\m(y)}{(\b)^{\circ}_{\m\n}(\e)^{\circ}_{\n}(\e)^{\circ}_{\m}}  \\
&=(1-q) \sum_{\m \in \widehat{\kappa^{*}}}f(\m)
\cfrac{  \p{\b\oa + \b\og}{\m}  }{  \pc{\b + \e}{\m}  } \m(y) 
{}_{2}F_{1} \left( 
\begin{matrix}
\a, \g \\
\b\m
\end{matrix}
;
x 
\right). \\
\end{align*}
By Lemma \ref{euler2} below, the right hand side is equal to  
\begin{align*}
&(1-q) \sum_{\m \in \widehat{\kappa^{*}} }f(\m) \cfrac{  \p{\b\oa + \b\og}{\m}  }{  \pc{\b + \e}{\m}  }  \m(y) 
\overline{\a\g}\b\m(1-x)
{}_2F_1\left( 
\begin{matrix}
\oa\b\m, \og\b\m \\
\b\m
\end{matrix}
;
x 
\right) \\
&+
(1-q)^2 \sum_{\m \in \widehat{\kappa^{*}}   }f(\m)\cfrac{  \p{\b\oa + \b\og}{\m}  }{  \pc{\b + \e}{\m}  }  \m(y) 
\left(\delta(\oa\b\m)\cfrac{g(\oa\g)}{g(\g) g(\oa) } \oa(x)  \right. \\
& \hspace{70mm}   \left. +\delta(\og\b\m) \cfrac{ g(\a\og) }{ g(\a) g(\og)}\og(x) \right).
\end{align*}
The first term is computed as follows:
\begin{align*}
&\overline{\a\g}\b(1-x)
\sum_{\m \in \widehat{\kappa^{*}} }f(\m)  \cfrac{  \p{\b\oa + \b\og}{\m} }{ \pc{\b + \e}{\m} }
\sum_{\n \in \widehat{\kappa^{*}} } \cfrac{  \p{\oa\b\m + \og\b\m}{\n} }{ \pc{\b\m + \e}{\n} }
\n(x)\m(1-x)\m(y) \\
&= 
\overline{\a\g}\b(1-x)
\sum_{\n, \m \in \widehat{\kappa^{*}} }f(\m)\cfrac{   \p{\b\oa + \b\og}{\m\n} }{ \pc{\b}{\m\n} }
\cfrac{\n(x)\m(y(1-x))}{ \pc{\e}{\m}  \pc{\e}{\n} }.
\end{align*}
By substituting $\eta\om$ for $\n$ and using the formulas \eqref{prelection} and \eqref{transitivity}, we obtain 
\begin{align*}
%&\overline{\a\g}\b(1-x)
% \sum_{\m, \eta}f(\m)\cfrac{(\b\oa)_{\eta}(\b\og)_{\eta}}{(\b)^{\circ}_{\eta}}
% \cfrac{\eta\om(x)\m(y(1-x))}{(\e)^{\circ}_{\m}(\e)^{\circ}_{\eta\om}} \\
%&=
&\overline{\a\g}\b(1-x)
\sum_{\eta \in \widehat{\kappa^{*}}} \cfrac{ \p{\b\oa + \b\og}{\eta} }{\pc{\b + \e}{\eta}}
\eta(x)
\sum_{\m \in \widehat{\kappa^{*}}}
f(\m)   \cfrac{1}{   \pc{\e}{\m}  \pc{\eta}{\om}}
\m\left(\cfrac{y(1-x)}{x}\right) \\
&=  \overline{\a\g}\b(1-x)
\sum_{\eta \in \widehat{\kappa^{*}}}\cfrac{ \p{\b\oa + \b\og}{\eta} }{\pc{\b + \e}{\eta}}  \eta(x)
\sum_{\m \in \widehat{\kappa^{*}}}  f(\m)  \cfrac{\p{\oeta}{\m}}{\pc{\e}{\m}} \m\left(\cfrac{y(x-1)}{x}\right).
 \end{align*}
On the other hand, the second term is computed as follows:
\begin{align*}
%&\overline{\a\g}\b(1-x)
% \sum_{\m}f(\m)\cfrac{(\b\oa)_{\m}(\b\og)_{\m}}{ \pc{\b}{\m} \pc{\e}{\m}}\m(y) 
 %\left( (1-q) \delta(\oa\b\m)\cfrac{g(\g\oa)}{g(\g)g(\oa)} \oa(x) \right)\\
 %&=
& (1-q)^2 f(\a\ob)
\cfrac{\p{\b\oa + \b\og}{\a\ob} }{\pc{\b + \e}{\a\ob}} \a\ob(y) \cdot \cfrac{g(\g\oa)}{g(\g) g(\oa)}
\oa(x) \\
&=
(1-q)^2 f(\a\ob)\cfrac{g(\e)g(\a\og)\gc{\b}\gc{\e}g(\g\oa)}{g(\b\oa)g(\b\og)\gc{\a}\gc{\a\ob}g(\g)g(\oa)}
 \ob(y) \a\left(\cfrac{y}{x} \right) \\
 &=(1-q)^2 f(\a\ob) \cfrac{\pc{\e}{\b}}{\p{\og}{\b}} \ob(-y) \a\left(-\cfrac{y}{x}\right). 
\end{align*}
%% The third term is same, hence we obtain the theorem.
By the symmetry of $\a$ and $\g$, the third term equals
\begin{align*}
(1-q)^2 f(\g\ob) \cfrac{\pc{\e}{\b}}{\p{\oa}{\b}} \ob(-y) \g\left(-\cfrac{y}{x}\right),
\end{align*}  
hence we obtain the theorem.
\end{proof}

\begin{lemm} \label{euler2}
Assume $(\a+\g+\overline{\a}\g, \e)=0$. Then for any $\mu \in \widehat{\kappa^{*}}$ and $x \in \kappa \backslash \{ 1 \}$, we have
\begin{align*}
{}_{2}F_{1} \left( 
\begin{matrix}
\a, \g \\
\b\m
\end{matrix}
;
x 
\right) 
&=\overline{\a\g}\b\mu(1-x)
{}_{2}F_{1} \left( 
\begin{matrix}
\oa\b\mu,  \og\b\mu \\
\b\mu
\end{matrix}
;
x
\right) \\
&+(1-q)\left(\delta(\oa\b\m)\cfrac{g(\oa\g)}{g(\g)g(\oa)} \oa(x)+\delta(\og\b\m)\cfrac{g(\a\og)}{g(\a)g(\og)}\og(x) \right).
\end{align*}
\end{lemm}
\begin{proof}
It is obvious when $x=0$, thus we may assume $x \neq 0$. 
The formula of the case $(\a + \g, \e + \b\m) =0$, that is, $\m \neq \a\ob, \g\ob$ follows from \eqref{euler0}. 
If $\m = \a \ob$ (resp. $\m=\g\ob$),  both sides coincide with 
\begin{align*}
\og(1-x)+ \cfrac{g(\g\oa)}{g(\oa)g(\g)} \oa(x)\,\,\,\,\,\,
\left(resp. \,\,\,\oa(1-x)+ \cfrac{g(\a\og)}{g(\og)g(\a)} \og(x)\right)
\end{align*}
by \eqref{reflection}, \eqref{shift}, \eqref{cancellation} and \eqref{geom}.
Therefore, we obtain the lemma. 
\end{proof}

As an application of Theorem \ref{thirdtrans}, we have the following analogous formulas of \eqref{cor3.2} and \eqref{cor3.3}. 

\begin{cor}
Assume $(\a+\g+\oa\g, \e)=0$.   
\begin{enumerate}
\item  For any $x, y \neq 0$ with $x \neq 0, 1$, we have
\begin{align*}
&F_{1: 0:0}^{0:2:2} \left( \left. 
\begin{matrix}
\hspace{1mm} \\
\b
\end{matrix}
;
\begin{matrix}
\a, \g  \\
\hspace{1mm}
\end{matrix}
;
\begin{matrix}
\b\oa, \b\og \\
\hspace{1mm}
\end{matrix} 
\right| 
x,  y
\right)  = \b\overline{\a\g}(1-x) \gh{\oa\b,  \og \b}{\b}{x+y-xy} \\
& + \b\overline{\a\g} (1-x) \delta (x+y-xy)   
+\ob(-y)\left( 
\cfrac{\pc{\e}{\b}}{\p{\og}{\b}} \a\left(-\cfrac{y}{x} \right)
+\cfrac{\pc{\e}{\b}}{\p{\oa}{\b}}  \g\left(-\cfrac{y}{x}\right)
\right).
\end{align*}

\item  Assume $(\varphi, \e + \rho)=0$. For any $x \in \kappa \backslash \{ 0, 1 \}$, we have
\begin{align*}
&F_{1:0:1}^{0:2:3} \left( \left. 
\begin{matrix}
\hspace{1mm} \\
\b
\end{matrix}
;
\begin{matrix}
\a, \g  \\
\hspace{1mm}
\end{matrix}
;
\begin{matrix}
\b\oa, \b\og, \varphi \\
\rho
\end{matrix} 
\right| 
x,  \frac{x}{x-1}
\right)  
=
 \b\overline{\a\g}(1-x) {}_{3}F_{2} \left( 
\begin{matrix}
\oa\b,    \b\og,  \overline{\varphi}\rho  \\
\b,  \rho
\end{matrix}
; x 
\right) \\
& \hspace{15mm}+\ob\left( \cfrac{x}{1-x}\right)
\left(
\cfrac{ \p{\varphi}{\a\ob} \pc{\e}{\b}}{  \pc{\rho}{\a\ob} \p{\og}{\b}}\oa\left(1-x\right)
 +
\cfrac{ \p{\varphi}{\g\ob}  \pc{\e}{\b}}{ \pc{\rho}{\g\ob} \p{\oa}{\b}}\og\left(1-x\right)
\right).
 \end{align*}

\end{enumerate}
\end{cor}

\begin{proof} 
(i) By applying Theorem \ref{thirdtrans} to $f(\m) \equiv 1$, we have 
\begin{align*}
F_{1: 0:0}^{0:2:2} \left( \left. 
\begin{matrix}
\hspace{1mm} \\
\b
\end{matrix}
;
\begin{matrix}
\a, \g  \\
\hspace{1mm}
\end{matrix}
;
\begin{matrix}
\b\oa, \b\og \\
\hspace{1mm}
\end{matrix} 
\right| 
x,  y
\right)  
&=
%\cfrac{\b\overline{\a\g}(1-x)}{(1-q)^2} \sum_{\eta}\cfrac{(\b\oa)_{\m}(\b\og)_{\m}}{(\e)^{\circ}_{\eta}(\b)^{\circ}_{\eta}} \eta(x)
\cfrac{\b\overline{\a\g}(1-x)}{(1-q)} \sum_{\eta \in \widehat{\kappa^{*}}}
\cfrac{ \p{\b\oa + \b\og}{\eta}  }{  \pc{\b + \e}{\eta} } \eta(x)
{}_{1}F_{0} \left( 
\begin{matrix}
\oeta  \\
\hspace{1mm}
\end{matrix}
; \cfrac{y(x-1)}{x}
\right)\\
%\sum_{\m} \cfrac{(\overline{\eta})_{\m}}{(\e)^{\circ}_{\m}} \left(\cfrac{y(x-1)}{x}\right) \\
&+ \ob(-y)\left(  \cfrac{\pc{\e}{\b}}{\p{\og}{\b}} \a\left(-\cfrac{y}{x} \right)
+\cfrac{\pc{\e}{\b}}{\p{\oa}{\b}}\g\left(-\cfrac{y}{x}\right)
\right). \\
\end{align*}
Note that we assume $x \neq 1$ and $y \neq 0$. By \eqref{geom},  
the first term agrees with 
\begin{align*}
&\cfrac{\overline{\a\g}\b(1-x)}{1-q}
\sum_{\eta  \neq \e} \cfrac{ \p{\b\oa + \b\og}{\eta}  }{  \pc{\b  + \e}{\eta}  }  \eta(x)
\eta\left( 1- \cfrac{y(x-1)}{x}\right) \\
&\hspace{5mm} + \cfrac{\b\overline{\a\g}(1-x)  \e(x)}{1-q} \left(-q \delta \left( 1-\cfrac{y(1-x)}{x}\right) +  1 \right) \\
&=\cfrac{\overline{\a\g}\b(1-x)}{1-q}
 \sum_{\eta \in \widehat{\kappa^{*}} } \cfrac{  \p{\b\oa + \b\og}{\eta} }{ \pc{\b + \e}{\eta} } \eta(x+y-xy) \\
&\hspace{5mm}- \cfrac{\b\overline{\a\g}(1-x)  \e(x)}{1-q}
\left( \e \left( x + y -xy \right) +  q\delta \left(x + y -xy \right) - 1 \right) \\
&=\b\overline{\a\g}(1-x) \gh{\oa\b,  \og \b}{\b}{x+y-xy} + \b\overline{\a\g} (1-x) \delta (x+y-xy).
 \end{align*}
This proves the formula.
%The latter part is $1-q$ when $y(x-1)/x=1$, otherwise 0. Therefore RHS is 
%\begin{align*}
%& \b\overline{\a\g}(1-x) \gh{\oa\b,  \og \b}{\b}{x+y-xy} \\
%& + \b\overline{\a\g} (1-x) \delta (x+y-xy)   
%+\ob(-y)\left( \cfrac{\g(-1)}{j(\b\og, \g)} \a\left(-\cfrac{y}{x}\right) +  \cfrac{\a(-1)}{j(\b\oa, \a)} \g\left(-\cfrac{y}{x}\right) \right).
%\end{align*}

(ii) By applying Theorem \ref{thirdtrans} to $f(\m) = \p{\varphi}{\m}/ \pc{\rho}{\m}$ and putting $y= x / (x-1)$, 
we have
\begin{align*}
F_{1:0:1}^{0:2:3} \left( \left. 
\begin{matrix}
\hspace{1mm} \\
\b
\end{matrix}
;
\begin{matrix}
\a, \g  \\
\hspace{1mm}
\end{matrix}
;
\begin{matrix}
\b\oa, \b\og, \varphi \\
\rho
\end{matrix} 
\right| 
x,  \frac{x}{x-1}
\right)  
&=\cfrac{\b\overline{\a\g}(1-x)}{1-q} \sum_{\eta \in \widehat{\kappa^{*}}}
\cfrac{  \p{\b\oa + \b\og}{\eta}  }{  \pc{\b+ \e}{\eta} } \eta(x)
{}_{2}F_{1} \left( 
\begin{matrix}
\oeta, \varphi  \\
\rho
\end{matrix}
; 1
\right)\\
%\sum_{\m} \cfrac{(\overline{\eta})_{\m}}{(\e)^{\circ}_{\m}} \left(\cfrac{y(x-1)}{x}\right) \\
&+\ob\left( \cfrac{x}{1-x}\right)
\left(
\cfrac{ \p{\varphi}{\a\ob} \pc{\e}{\b}}{  \pc{\rho}{\a\ob} \p{\og}{\b}}\oa(1-x) 
+ \cfrac{ \p{\varphi}{\g\ob}  \pc{\e}{\b}}{ \pc{\rho}{\g\ob} \p{\oa}{\b}}\og(1-x)
\right).
\end{align*}
By Vandermonde's formula \eqref{Vande}, we have
\begin{align*}
 {}_{2}F_{1} \left( 
\begin{matrix}
\overline{\eta}, \varphi  \\
\rho
\end{matrix}
; 1
\right)
=
\cfrac{(\overline{\varphi} \rho)_{\eta}}{(\rho)^{\circ}_{\eta}},
\end{align*}
hence we obtain 
\begin{align*}
\cfrac{\b\overline{\a\g}(1-x)}{1-q} 
\sum_{\eta \in \widehat{\kappa^{*}} } \cfrac{ \p{\b\oa + \b\og}{\eta} }{ \pc{\b + \e}{\eta} } 
\eta(x)
{}_{2}F_{1} \left( 
\begin{matrix}
\overline{\eta}, \varphi  \\
\rho
\end{matrix}
; 1
\right)
=
\b\overline{\a\g}(1-x) {}_{3}F_{2} \left( 
\begin{matrix}
\oa\b,    \b\og,  \overline{\varphi}\rho  \\
\b,  \rho
\end{matrix}
; x 
\right).
\end{align*}
This proves the formula.
\end{proof}

%%%%%%%%%%%%%%%%%%%%%%%%%%% Nemoto's part %%%%%%%%%%%%%%%%%%%%%%%%%%%%%%%%

%%%%%%%%%%%%%%%%%%%%%%%%%%% Kumabe's part %%%%%%%%%%%%%%%%%%%%%%%%%%%%%%%%

%In this section, we give reduction formulas which are analogues of \cite[Corollary 3.4]{liuwang} over finite fields. To %obtain our results, we slightly extend \cite[Theorem 3.14 (ii)]{otsubo} which is an analogue of the formula due to Pfaff.  

Next, we give an analogous formula of \cite[Theorem 3.3]{liuwang}:
\begin{align*}
\begin{split}
&\sum_{i,j=0}^{\infty} \Omega(j) \frac{\p{a}{i} \p{c}{i} \p{b - c}{j} }{\p{b}{i+j}} \frac{x^{i} y^{j}}{i!j!} \\
&= (1-x)^{-a} \sum_{n=0}^{\infty} \frac{\p{a}{n} \p{b - c}{n} }{\p{b}{n} n! } \left( \frac{x}{x-1}\right)^{n}
\sum_{j=0}^{n} \Omega (j) \frac{\p{-n}{j}}{j! \p{1-a-n}{j}} \left( \frac{y(x-1)}{x} \right)^{j},
\end{split}
\end{align*}
where $\Omega (j)$ is a complex sequence.

%The following result is an analogue of \cite[Theorem 3.3]{liuwang}.

\begin{theo}\label{fourthtrans}
Assume $(\a + \g, \e) = (\a, \g)=0$. For any $x, y \in \kappa$ with $x \neq 0, 1$,  we have 
\begin{equation*}
\begin{split}
\sum_{\n, \m \in \widehat{\kappa^{*}}}&f(\m)\frac{\p{\a}{\n} \p{\g}{\n} \p{\b\og}{\m}}{\pc{\b}{\n\m}} \frac{\n(x)\m(y)}{\pc{\e}{\n} \pc{\e}{\m}} \\
&= \oa(1-x)\sum_{\eta \in \widehat{\kappa^{*}}} 
\frac{\p{\a + \b\og}{\eta} }{  \pc{ \b + \e}{\eta} } \eta\left(\frac{x}{x-1} \right) 
\sum_{\m \in \widehat{\kappa^{*}}}
f(\m)\frac{\p{\oeta}{\m}}{\pc{\ol{\a\eta} + \e}{\m}} \m\left(\frac{y(x-1)}{x} \right) \\
&+(1-q)^2f(\ob\g)\pc{\e}{\b}\p{\a}{\og}\ob(-y)\g\left(\frac{y}{x}\right).
\end{split}
\end{equation*}
\end{theo}
%二変数のフーリエ変換でも証明できるか？
\begin{proof}
By Lemma \ref{Pfaff} below, the left hand side of the formula above coincides with
\begin{equation*}
\begin{split}
(1-q)&\sum_{\m \in \widehat{\kappa^{*}}} 
f(\m)\frac{\p{\b\og}{\m}}{\pc{\b + \e}{\m}} \m (y)  \\
&\times\left(\oa(1-x)\gh{\a,\og\b\m}{\b\m}{ \frac{x}{x-1} } + \delta(\m\b\og)\frac{g(\a\og)}{g(\a)g(\og)}\og(x)(1-q) \right).
\end{split}
\end{equation*}
By \eqref{transitivity} and putting $\eta$ to $\m\n$, we obtain
\begin{equation*}
\begin{split}
&(1-q)\sum_{\m \in \widehat{\kappa^{*}}}
f(\m)\frac{\p{\b\og}{\m}}{\pc{\b + \e}{\m} } \m(y) 
\oa(1-x)\gh{\a,\og\b\m}{\b\m}{ \frac{x}{x-1} } \\
&=\oa(1-x)\sum_{\m, \n \in \widehat{\kappa^{*}}}
f(\m)\frac{\p{\b\og}{\m\n}\p{\a}{\n}}{\pc{\b}{\m\n}\pc{\e}{\m}\pc{\e}{\n}}\m(y)\n \left(\frac{x}{x-1}\right) \\
&=\oa(1-x)\sum_{\eta, \m  \in \widehat{\kappa^{*}} }f(\m)\frac{\p{\a}{\om \eta}\p{\b\og}{\eta}}{\pc{\b}{\eta}\pc{\e}{\m}\pc{\e}{\om\eta}}\m(y)\om\eta\left(\frac{x}{x-1}\right) \\
&=\oa(1-x)\sum_{\eta \in \widehat{\kappa^{*}}}\frac{\p{\a + \b\og}{\eta} }{\pc{\b+ \e}{\eta} } \eta\left(\frac{x}{x-1}\right)
\sum_{\m \in \widehat{\kappa^{*}}}f(\m)\frac{\p{\oeta}{\m}}{\pc{\ol{a\eta} + \e}{\m} }  \m\left(\frac{y(x-1)}{x}\right).
\end{split}
\end{equation*}
By an elementary computation,  we can see that the remaining term coincides with
\begin{equation*}
(1-q)^2f(\ob\g)\pc{\e}{\b}\p{\a}{\og}\ob(-y)\g\left(\frac{y}{x}\right).
\end{equation*}
Hence the proof is complete.
\end{proof}

\begin{lemm}\label{Pfaff}
Suppose that $(\a + \g, \e) = (\a,  \g)= 0$. 
Then for any $\m \in \widehat{\kappa^{*}}$ and $x \in \kappa \backslash \{ 1 \}$,  we have
\begin{equation*}
\begin{split}
\gh{\a, \g}{\b\m}{x} = \oa(1-x)&\gh{\a, \b\og\mu}{\b\m}{\frac{x}{x-1}} 
+ (1-q)\delta (\m\b\og)\frac{g(\a\og)}{g(\a)g(\og)}\og(x).
\end{split}
\end{equation*}
\end{lemm}
\begin{proof}
If $(\a + \g, \e + \b\m) = 0$, then the argument is just Pfaff's formula \eqref{pfaff0}.
If $\m = \g\ob$ (resp. $\m = \a\ob$), by similar computation to the proof of Lemma \ref{pfaff1}, both sides 
coincide with
\begin{align*}
\oa(1-x) + \frac{g(\a\og)}{ g(\a)g(\og) }\og(x). \hspace{6mm}
\left(
resp.  \hspace{2mm} \og(1-x) + \frac{g(\oa\g)}{  g(\g)  g(\oa)  }  \oa(x).
\right)
\end{align*}
%We consider the case $\m=\g\ob$. By (\ref{shift}), (\ref{cancellation}) and (\ref{geom}), both sides coincide with 
%\begin{equation*}
%\oa(1-x) + \frac{g(\a\og)}{ g(\a)g(\og) }\og(x)
%\end{equation*}
%by (\ref{shift}), (\ref{cancellation}) and (\ref{geom}).
%The case $\m = \a\ob$ is similar to the case $\m=\g\ob$. We can see that both sides of the above formula coincide with
%\begin{equation*}
%\og(1-x) + \frac{g(\oa\g)}{  g(\g)  g(\oa)  }  \oa(x).
%\end{equation*}
This proves the lemma.
\end{proof}

As an application of Theorem \ref{fourthtrans}, we obtain the following analogous formulas of \eqref{cor3.6},  \eqref{cor3.7} and \eqref{cor3.9}. 

\begin{cor}\label{fourthcor}
\begin{enumerate}
\item \label{2f1} Asuume $(\a + \g, \e) = (\a, \g + \ol{\sigma})= 0$. For any $x \in \kappa \backslash \{ 0, 1 \}$, we have 
\begin{align*}
\begin{split}
F_{1: 0:0}^{0:2:2} \left( \left. 
\begin{matrix}
\hspace{1mm} \\
\b
\end{matrix}
;
\begin{matrix}
\a, \g \\
\hspace{1mm}
\end{matrix}
;
\begin{matrix}
\b\og, \sigma \\
\hspace{1mm}
\end{matrix} 
\right| 
x,  \frac{x}{x-1}
\right) 
&=\oa(1-x)\gh{\b\og, \a\sigma}{\b}{\frac{x}{x-1}} \\
&+\p{\sigma}{\ob\g}\pc{\e}{\b}\p{\a}{\og}\ob\left(\frac{x}{1-x}\right)\og(x-1).
\end{split}
\end{align*}
\item \label{3f2} Assume $(\a + \g + \sigma, \e) = (\a,\g+ \ol{\sigma}) = (\tau, \e+\oa)=0$. For any $x \in \kappa \backslash \{ 0, 1 \}$, we have 
\begin{equation*}
\begin{split}
F_{1:0:1}^{0:2:3} &\left( \left. 
\begin{matrix}
\hspace{1mm} \\
\b
\end{matrix}
;
\begin{matrix}
\a, \g  \\
\hspace{1mm}
\end{matrix}
;
\begin{matrix}
\b\og, \sigma, \tau \\
\a\sigma\tau
\end{matrix} 
\right| 
x,  \frac{x}{x-1}
\right) \\
=\oa&(1-x)\left(
{_{3}F_{2}}\left(
\begin{matrix}
\b\og, \a\sigma, \a\tau \\
\b, \a\sigma\tau
\end{matrix}
; \frac{x}{x-1}
\right) \right. \\
 +&\left.\frac{\gc{\a\sigma\tau}g(\og)\gc{\b}}{g(\a)g(\sigma)g(\tau)g(\b\og)}\a(-1)\b\left(\frac{x-1}{x}\right)\og(1-x)
\right)\\
&+\frac{\p{\sigma}{\ob\g}\p{\tau}{\ob\g}\pc{\e}{\b}\p{\a}{\og}}{\pc{\a\sigma\tau}{\ob\g}}\ob\left(\frac{x}{1-x} \right)\og(x-1).
\end{split}
\end{equation*}

\item Assume $(\a^{2} + \g + \sigma + \a^{2} \sigma, \e) = (\a^{2}, \g ) =0$. For any $x\in \kappa \backslash \{0, 1\}$, we have 
\begin{align*}
F_{1:0:1}^{0:2:3} \left( \left. 
\begin{matrix}
\hspace{1mm} \\
\b
\end{matrix}
;
\begin{matrix}
\a^{2}, \g  \\
\hspace{1mm}
\end{matrix}
;
\begin{matrix}
\b\og, \sigma, \oa \\
\a\sigma
\end{matrix} 
\right| 
x,  \frac{x}{x-1}
\right)
&= \oa^{2} (1-x) {_{3}F_{2}}\left(
\begin{matrix}
\b\og, \a^{2}\sigma, \a \\
\b, \a\sigma
\end{matrix}
; \frac{x}{x-1}
\right) \\
&+\ol{\a^{2}\g}\b (1-x) \ob(-x) q^{\delta (\a\sigma)} \frac{\p{\sigma}{\a} \pc{\e}{\b} }{\p{\oa^{2}}{\a} \p{\og}{\b}}  \\
&+\ob(x) \b\og(1-x) \frac{\p{\sigma}{\ob\g} \p{\oa}{\ob\g} \pc{\e}{\b}}{\pc{\a\sigma}{\ob\g} \pc{\oa^{2}}{\g}}.
\end{align*}
\end{enumerate}
\end{cor}

\begin{proof}
%The equality in (\ref{2f1}) 
(i) By applying Theorem \ref{fourthtrans} to $f(\m)=\p{\sigma}{\m}$ and putting $y = x/(x-1)$, the left hand side in loc. cit.  is equal to
\begin{equation*}
(1-q)^2F_{1: 0:0}^{0:2:2} \left( \left. 
\begin{matrix}
\hspace{1mm} \\
\b
\end{matrix}
;
\begin{matrix}
\a, \g \\
\hspace{1mm}
\end{matrix}
;
\begin{matrix}
\b\og, \sigma \\
\hspace{1mm}
\end{matrix} 
\right| 
x,  \frac{x}{x-1}
\right),
\end{equation*}
and the first and second terms on the right hand side are equal to 
\begin{align*}
&(1-q)^2\oa(1-x)\gh{\b\og, \a\sigma}{\b}{\frac{x}{x-1}}, \\
&(1-q)^2\p{\sigma}{\ob\g}\pc{\e}{\b}\p{\a}{\og}\ob\left(\frac{x}{1-x}\right)\og(x-1),
\end{align*}
respectively. Here we used \eqref{eulergauss} to compute the first term. This proves the formula.

%By the Euler-Gauss formula \eqref{eulergauss}, we can see that the first term on the right hand side of Theorem \ref{fourthtrans} is equal to
%\begin{equation*}
%(1-q)^2\oa(1-x)\gh{\b\og, \a\sigma}{\b}{\frac{x}{x-1}}.
%\end{equation*}
%The second term on the right hand side of Theorem \ref{fourthtrans} is equal to
%\begin{equation*}
%\begin{split}
%&(1-q)^2\p{\sigma}{\ob\g}\ob\g\left(\frac{x}{y(x-1)}\right)\pc{\e}{\b}\p{\a}{\og}\ob(-y)\g\left(\frac{y}{x}\right) \\
%&=(1-q)^2\p{\sigma}{\ob\g}\pc{\e}{\b}\p{\a}{\og}\ob\left(\frac{x}{1-x}\right)\og(x-1).
%\end{split}
%\end{equation*}

(ii) We apply Theorem \ref{fourthtrans} to $f(\m) =\p{\sigma + \tau}{\m} / \pc{\a\sigma\tau}{\m}$ and put $y = x/(x-1)$. 
The left hand side coincides with 
\begin{equation*}
(1-q)^2F_{1:0:1}^{0:2:3} \left( \left. 
\begin{matrix}
\hspace{1mm} \\
\b
\end{matrix}
;
\begin{matrix}
\a, \g  \\
\hspace{1mm}
\end{matrix}
;
\begin{matrix}
\b\og, \sigma, \tau \\
\a\sigma\tau
\end{matrix} 
\right| 
x,  \frac{x}{x-1}
\right) . 
\end{equation*}
By the Saalsch\"utz formula \eqref{saalschutz} and the formulas \eqref{shift} and \eqref{geom},  we can see that the first term on the right hand side is equal to
\begin{equation*}
\begin{split}
&(1-q)^2 \oa(1-x)\left(
{}_{3}F_2\left(
\begin{matrix}
\b\og, \a\sigma, \a\tau \\
\b, \a\sigma\tau
\end{matrix}
; \frac{x}{x-1}
\right) \right. \\
&+\left.\frac{\gc{\a\sigma\tau}g(\og)\gc{\b}}{g(\a)g(\sigma)g(\tau)g(\b\og)}\a(-1)\b\left(\frac{x-1}{x}\right)\og(1-x) 
\right).
\end{split}
\end{equation*}
The second term on the right hand side is equal to 
\begin{equation*}
(1-q)^2\frac{\p{\sigma}{\ob\g}\p{\tau}{\ob\g}\pc{\e}{\b}\p{\a}{\og}}{\pc{\a\sigma\tau}{\ob\g}}\ob\left(\frac{x}{1-x}\right)\og(x-1).
\end{equation*}
Therefore we obtain the formula.

(iii) If we apply Theorem \ref{fourthtrans} to $f(\m) =\p{\sigma+ \oa}{\m} / \pc{\a\sigma}{\m}$ and put $y = x/(x-1)$, the left hand side coincides with 
\begin{equation*}
(1-q)^{2} F_{1:0:1}^{0:2:3} \left( \left. 
\begin{matrix}
\hspace{1mm} \\
\b
\end{matrix}
;
\begin{matrix}
\a^{2}, \g  \\
\hspace{1mm}
\end{matrix}
;
\begin{matrix}
\b\og, \sigma, \oa \\
\a\sigma
\end{matrix} 
\right| 
x,  \frac{x}{x-1}
\right).
\end{equation*}
By the Saalsch\"utz formula \eqref{saalschutz} and the formulas \eqref{shift} and \eqref{geom},  we can see that the first term on the right hand side is equal to
\begin{align*}
(1-q)^2 \oa^{2} (1-x)  \left(
{}_{3}F_2\left(
\begin{matrix}
\b\og, \a^{2}\sigma, \a \\
\b, \a\sigma
\end{matrix}
; \frac{x}{x-1}
\right) 
+ \ob(-x) \b \og(1-x) q^{\delta (\a\sigma)} \frac{ \p{\sigma}{\a} \pc{\e}{\b} }{\p{\oa^{2}}{\a} \p{\og}{\b}}  
\right).
\end{align*}
The second term on the right hand side agrees with 
\begin{align*}
\ob(x) \b\og(1-x) \frac{\p{\sigma}{\ob\g} \p{\oa}{\ob\g} \pc{\e}{\b}}{\pc{\a\sigma}{\ob\g} \pc{\oa^{2}}{\g}}.
\end{align*}
Hence the proof is complete. 
\end{proof}

If we specialize parameters and arguments, then we obtain the following analogous formula of \cite[Corollary 5.4]{liuwang}.

\begin{cor}
Assume $p \neq 2$. 
\begin{enumerate}
\item Suppose that $(\a + \g^2,\e)=0$. Then, we have
\begin{align*}
&\kF{0:2:2}{1:0:0}{}{\b^2}{\a,\g^2}{}{\b^2\og^2, \overline{\a\g^2}}{}{\dfrac{1}{2},-1} \\
&=\a(2)\sum_{\chi^2=\e}\dfrac{\gc{\b^2}g(\b\og\chi)}{g(\b^2\og^2)\gc{\b\g\chi}}+\g(4) \p{\ol{\a\g^{2}}}{\ob^{2}\g^{2}} \pc{\e}{\b^2} \p{\a}{\og^2} .
\end{align*}

\item Suppose that $(\a + \g^2+\b^2\g^2+\b^2\og^2,\e)=0$. Then, we have
\begin{align*}
&\kF{0:2:2}{1:0:0}{}{\b^2}{\a,\g^2}{}{\b^2\og^2,\oa\b^2\g^2}{}{-1,\dfrac{1}{2}}\\
&=\oa(2)\sum_{\chi^2=\e}\dfrac{\gc{\b^2}g(\phi)}{g(\phi\b\og\chi)g(\b\g\chi)}+\b\og(4)\p{\oa\b^2\g^2}{\ob^2\g^2}\pc{\e}{\b^2}\p{\a}{\og^2}.
\end{align*}

\item Suppose that $(\a + \g^{2} + \ob^2\g^2+\b^4\og^2,\e)=0$. Then, we have
\begin{align*}
&\kF{0:2:2}{1:0:0}{}{\b^2}{\a,\g^2}{}{\b^2\og^2,\overline{\a\b^2}\g^2}{}{-1,\dfrac{1}{2}}\\
&=\oa(2)\sum_{\chi^2=\e}\dfrac{\gc{\b}\gc{\phi\b}}{g(\b^2\og\chi) g(\phi\g\chi)}+\b\og(4)\p{\overline{\a\b^2}\g^2}{\ob^2\g^2}\pc{\e}{\b^2}\p{\a}{\og^2}.
\end{align*}
\end{enumerate}
\end{cor}

\begin{proof}
Part (i) (resp. (ii), (iii)) follows from Corollary \ref{fourthcor} (i) and \eqref{kummer} (resp. \eqref{gsecond}, \eqref{bailey}). 
%\begin{enumerate}
%\item Using Corollary \ref{fourthcor} (i) and \eqref{kummer}, we obtain (i).
%
%\item Using Corollary \ref{fourthcor} (i) and \eqref{gsecond}, we obtain (ii).
%
%\item Using Corollary \ref{fourthcor} (i) and \eqref{bailey}, we complete the proof.
%\end{enumerate}
\end{proof}

%%%%%%%%%%%%%%%%%%%%%%%%%%% Kumabe's part %%%%%%%%%%%%%%%%%%%%%%%%%%%%%%%%

\subsection{Formulas for $F_{1: 0:C^{\prime}}^{2:0:C} $}

%%%%%%%%%%%%%%%%%%%%%%%%%%% Ito's part3 %%%%%%%%%%%%%%%%%%%%%%%%%%%%%%%%%%%

We give a finite field analogue of \cite[Theorem 4.1]{liuwang}:
\begin{align*}
\begin{split}
&\sum_{i,j=0}^{\infty} \Omega(j) \frac{\p{a}{i+j} \p{c}{i+j}  }{\p{b}{i+j}} \frac{x^{i} y^{j}}{i!j!} \\
&= (1-x)^{-a} \sum_{n=0}^{\infty} \frac{\p{a}{n} \p{b - c}{n} }{\p{b}{n} n!} \left( \frac{x}{x-1}\right)^{n}
\sum_{j=0}^{n} \Omega (j) \frac{\p{-n}{j}  \p{c}{n} }{j! \p{1+c-b-n}{j} } \left( -\frac{y}{x} \right)^{j},
\end{split}
\end{align*}
where $\Omega (j)$ is a complex sequence.

\begin{theo}\label{fifthtrans}
Assume  $(\a,  \b + \g) = (\b,  \g) =0$.  Let $f\colon  \widehat{\kappa^{*}} \rightarrow \mathbb{C}$ be a map.   
For any $x, y \in \kappa$ with $x \neq 0, 1$,  we have 
\begin{align*}
&\sum_{\n,  \m \in \widehat{\kappa^{*}}} 
f(\m) \frac{\p{\a}{\n\m} \p{\g}{\n\m} }{\pc{\b}{\n\m} \pc{\e}{\n} \pc{\e}{\m}} \n (x) \m (y)  \\
&= \oa (1-x) \sum_{\eta \in \widehat{\kappa^{*}} } \frac{\p{\a + \b\og }{\eta}  }{\pc{\b + \e }{\eta}} \eta \left( \frac{x}{x-1} \right) 
\sum_{\m \in \widehat{\kappa^{*}} } f(\m) \frac{\p{\oeta + \g }{\m} }{\pc{\e + \g\ol{\b\eta}}{\m}  } \m \left( - \frac{y}{x} \right)
\\
& \hspace{10mm} - (1-q)^{2} f(\og) \frac{ \pc{\e}{\b} }{ \pc{\oa}{\b} } \ob \g (-x)  \og (y) \oa\b (1-x) .
\end{align*}
\end{theo}
\begin{proof}
If we use \eqref{transitivity},  the left hand side is equal to
\begin{align*} 
(1-q) \sum_{\m \in \widehat{\kappa^{*}}} f(\m)  \frac{\p{\a + \g}{\m} }{\pc{\b + \e }{\m} } \m (y) \gh{\a\m,  \g\m}{\b\m}{x} . 
\end{align*}
By applying Lemma \ref{euler3} below and simple computation,  we can see that the sum above equals
\begin{align*}
&\oa (1-x)\sum_{\n,  \m \in \widehat{\kappa^{*}}} 
f(\m) \frac{\p{\a}{\m\n}  \p{\b\og}{\n}  \p{\g}{\m}  }{\pc{\b}{\m\n}  \pc{\e}{\n}  \pc{\e}{\m}}  \n \left( \frac{x}{x-1} \right) \m \left( \frac{y}{1-x} \right)  \\
&\hspace{20mm}- q^{-1}(1-q)^{2} f(\og) \frac{\p{\a + \g}{\og}  }{\pc{\b + \e}{\og}  } \og (y) \frac{g(\a\ob) g(\b\og)}{g(\a\og)} \ob\g (x) \oa\b(1-x) . 
\end{align*}
If we substitute $\eta\om$ for $\n$ and use \eqref{prelection},    the first term coincides with 
\begin{align*}
&\oa (1-x) \sum_{\eta \in \widehat{\kappa^{*}}} \frac{\p{\a + \b\og}{\eta}  }{\pc{\b + \e }{\eta}  } \eta \left( \frac{x}{x-1} \right) \sum_{\m \in \widehat{\kappa^{*}}} 
f(\m) \frac{\p{\og\b\eta}{\om} \p{\g}{\m} }{\pc{\e}{\m} \pc{\eta}{\om} } \m \left( - \frac{y}{x} \right) \\
&= \oa (1-x) \sum_{\eta \in \widehat{\kappa^{*}}} \frac{\p{\a + \b\og }{\eta}  }{\pc{\b + \e }{\eta} } \eta \left( \frac{x}{x-1} \right)
\sum_{\m \in \widehat{\kappa^{*}}} f(\m) \frac{\p{\oeta + \g }{\m}  }{\pc{\g\ol{\b\eta} + \e}{\m}  } \m \left( - \frac{y}{x} \right) .
\end{align*}
Finally,  if we simplify the second term by using \eqref{reflection},   the theorem follows.
\end{proof}

\begin{lemm}\label{euler3}
Assume $(\a,  \g + \b) = (\b,  \g) =0$. For any $\m \in \widehat{\kappa^{*}}$ and $x \in \kappa \backslash \{ 0, 1 \}$, 
we have
\begin{align*}
\gh{\a\m,  \g\m}{\b\m}{x} 
&= \oa\om (1-x) \gh{\a\m,  \b\og}{\b\m}{\frac{x}{x-1}} \\
& + \delta (\g\m) (1-q^{-1}) \frac{g (\a\ob) g(\b\og) }{g(\a\og)} \ob\g (x) \oa\b (1-x) .
\end{align*}
\end{lemm}
\begin{proof}
The formula in the case $(\a\m + \g\m, \e + \b\m) =0$,  that is,  $\m \neq \oa,  \og$ follows from \eqref{pfaff0}.  
If $\m = \oa$ (resp.  $\m = \og$),   by  \eqref{cancellation},  \eqref{reflection} and  \eqref{shift},  
we can show that both sides are equal to 
\begin{align*}
&\frac{ g(\oa\b) g(\ob\g) }{g(\oa\g)} \a\ob (x)  \b\og (1-x) + 1. \\
&\left( resp. \hspace{3mm}
\frac{g(\a\ob) g(\b\og) }{g(\a\og)} \ob\g(x) \oa\b (1-x) + 1.
\right)
\end{align*}
Therefore we obtain the lemma.
\end{proof}

We have the following finite field analogues of \eqref{cor4.2} and \eqref{cor4.4}.
Similarly to Corollary \ref{secondcor} (iii), we remark that the function ${}_{3}F_{2}(x)$ in \eqref{cor4.4} reduces to ${}_{2}F_{1}(x)$ in the case of finite fields since the complex parameters $1+b/2$ and $b/2$ agree with the same character.

\begin{cor}\label{fifthcor}
\begin{enumerate}
\item Assume $(\a + \e, \b + \g) = (\b,  \g) = (\varphi,  \e +  \ob\g ) = 0$. For $x \in \kappa \backslash \{ 0, 1 \}$, we have 
\begin{align*}
F_{1:0:1}^{2:0:1} \left( \left. 
\begin{matrix}
\a,  \g \\
\b
\end{matrix}
; \hspace{1mm} ; 
\begin{matrix}
\varphi \\
\b \varphi
\end{matrix}
\right| x ,  -x
\right)  
= \oa (1-x) \gh{\a,  \b\varphi\og}{\b\varphi}{\frac{x}{x-1}} + \frac{\pc{\varphi + \e}{\b} }{ \pc{\oa + \og}{\b}} \ob (-x).
\end{align*}

\item Assume $(\a + \e, \b^{2} + \g^{2}) = (\b^{2}, \g^{2}) =0$. For $x \in \kappa \backslash \{ 0, 1 \}$, we have 
\begin{align*}
F_{1:0:1}^{2:0:1} \left( \left. 
\begin{matrix}
\a,  \g^{2} \\
\b^{2}
\end{matrix}
; \hspace{1mm} ; 
\begin{matrix}
\g\ob \\
\g\b
\end{matrix}
\right| x ,  -x
\right)  
&= \oa (1-x) {}_{2}F_{1} \left( 
\begin{matrix}
\a, \b\og \\
\b\g
\end{matrix}
; \frac{x}{x-1} \right) 
+ \ob^{2} (x) \frac{\p{\ob\g}{\b^{2}} \pc{\e}{\b^{2}} }{\p{\og^{2}}{\b^{2}} \pc{\oa}{\b^{2}}}. 
\end{align*}

\end{enumerate}
\end{cor}
\begin{proof}
(i) By applying Theorem \ref{fifthtrans} to $f(\m) = \p{\varphi}{\m} / \pc{\b\varphi}{\m}$ and putting  $y = -x$,  we have 
\begin{align*}
&F_{1:0:1}^{2:0:1} \left( \left. 
\begin{matrix}
\a,  \g \\
\b
\end{matrix}
; \hspace{1mm} ; 
\begin{matrix}
\varphi \\
\b \varphi
\end{matrix}
\right| x ,  -x
\right)   \\
&= \frac{1}{(1-q)^{2}} 
\oa (1-x) \sum_{\eta \in \widehat{\kappa^{*}}} \frac{\p{\a + \b\og }{\eta}  }{\pc{\b + \e }{\eta} } \eta \left( \frac{x}{x-1} \right) 
\sum_{\m \in \widehat{\kappa^{*}}} \frac{\p{\varphi + \oeta + \g }{\m}   }{\pc{\b\varphi + \g\ol{\b\eta} + \e}{\m} } \m (1)\\
& \hspace{40mm} -  \frac{\p{\varphi}{\og} \pc{\e}{\b}   }{\pc{\b\varphi}{\og} \pc{\oa}{\b} }  \ob (-x) \oa\b(1-x) .
\end{align*}
The first term is equal to 
\begin{align}
\frac{1}{1-q} 
\oa (1-x) \sum_{\eta \in \widehat{\kappa^{*}}} \frac{\p{\a +\b\og }{\eta} }{\pc{\b + \e }{\eta} } \eta \left( \frac{x}{x-1} \right) 
{}_{3}F_{2} \left( 
\begin{matrix}
\varphi ,  \oeta,  \g \\
\b\varphi,  \g\ol{\b\eta}
\end{matrix}
; 1 \right) . \label{fifthref2}
\end{align}
Note that the above ${}_{3}F_{2} (1)$ is Saalsch{\"u}tzian and $\varphi + \oeta + \g \neq \e +   \b\varphi + \g\ol{\b\eta}$ under our assumptions.  Hence we can apply \eqref{saalschutz} and obtain  
\begin{align*}
{}_{3}F_{2} \left( 
\begin{matrix}
\varphi ,  \oeta,  \g \\
\b\varphi,  \g\ol{\b\eta}
\end{matrix}
; 1 \right)
&= \frac{\gc{\b\varphi} g( \og\varphi \b\eta)  g (\og\b)  g(\b\eta) }{g(\og\b\eta) \gc{\b} \gc{\b\varphi\eta} \gc{\og\b\varphi}}
+ \frac{\gc{\b\varphi} \gc{\g\ol{\b\eta}} }{g(\varphi)  g(\oeta) g(\g)}  \\
&= \frac{\p{\og\b\varphi}{\eta} \p{ \b}{\eta}  }{\pc{\b\varphi}{\eta} \p{\og\b}{\eta}} 
+  \frac{\pc{\varphi}{\b}  \pc{\e}{\eta}}{ \pc{\og}{\b} \p{\og\b}{\eta}} \b(-1) . 
\end{align*}
Here we used  $(\b,  \e + \g) = (\varphi,  \ob\g) =0$ in the second equality.
Therefore \eqref{fifthref2} equals
\begin{align*}
&\oa (1-x) {}_{3}F_{2} \left( 
\begin{matrix}
\a,  \b\varphi\og,  \b \\
\b,  \b\varphi
\end{matrix}
;
\frac{x}{x-1} 
\right) 
+ \oa (1-x) \frac{\pc{\varphi}{\b} }{ \pc{\og}{\b}} \b(-1) 
F\left( \a,  \b;  \frac{x}{x-1} \right)  \\
&= \oa (1-x) \left( 
\gh{\a,  \b\varphi\og}{\b\varphi}{\frac{x}{x-1}} + q^{-1} \frac{\p{\a + \b\varphi\og}{\ob}  }{\pc{\e + \b\varphi}{\ob} }
\ob \left( \frac{x}{x-1} \right) 
\right) \\
& \hspace{10mm} + \oa (1-x) \frac{\pc{\varphi}{\b} \p{\a}{\ob}}{ \pc{\og}{\b} \pc{\b}{\ob} } \b(-1)  \ob \left(\frac{x}{x-1} \right) 
{}_{1}F_{0} \left( 
\begin{matrix}
\a \ob \\
\hspace{1mm}
\end{matrix}
;  \frac{x}{x-1}
\right) \\
&= \oa (1-x) \gh{\a,  \b\varphi\og  }{\b\varphi}{\frac{x}{x-1}}  \\
& \hspace{15mm} +\oa\b (1-x) \ob (-x) q^{-1} \frac{\p{\a + \b\varphi\og}{\ob}  }{\pc{\e + \b\varphi}{\ob} } +  \frac{\pc{\varphi + \e}{\b} }{ \pc{\oa + \og}{\b}} \ob (-x).  
\end{align*}
Here we used the formulas \eqref{cancellation} and \eqref{shift} in the first equality  and the formulas \eqref{geom} and \eqref{reflection} in the  last equality.
By \eqref{reflection} and noting that $(\b,  \e) = (\varphi,   \g\ob + \e) =0$,    we can show
\begin{align*}
\oa\b (1-x) \ob (-x) q^{-1} \frac{\p{\a + \og\b\varphi}{\ob}  }{\pc{\e + \b\varphi}{\ob} }
=  \frac{\p{\varphi}{\og} \pc{\e}{\b}   }{\pc{\b\varphi}{\og} \pc{\oa}{\b} }  \ob (-x) \oa\b(1-x).
\end{align*}
Hence the formula follows.

(ii) If we apply Theorem \ref{fifthtrans} to $f(\m) = \p{\g\ob}{\m} / \pc{\g\b}{\m}$ and put $y=-x$, we obtain 
\begin{align*}
&F_{1:0:1}^{2:0:1} \left( \left. 
\begin{matrix}
\a,  \g^{2} \\
\b^{2}
\end{matrix}
; \hspace{1mm} ; 
\begin{matrix}
\g\ob \\
\g\b
\end{matrix}
\right| x ,  -x
\right)    \\
&= \frac{1}{1-q} 
\oa (1-x) \sum_{\eta \in \widehat{\kappa^{*}}} \frac{\p{\a + \b^{2}\og^{2} }{\eta}  }{\pc{\b^{2} + \e }{\eta} } \eta \left( \frac{x}{x-1} \right) 
{}_{3}F_{2} \left( 
\begin{matrix}
\g\ob, \g^{2}, \oeta \\
\g\b, \g^{2} \ol{\b^{2}\eta}
\end{matrix}
;  1
\right) \\
& - \ob^{2}(-x) \oa\b^{2}(1-x) \frac{\p{\g\ob}{\og^{2}} \pc{\e}{\b^{2}}}{\pc{\g\b}{\og^{2}} \pc{\oa}{\b^{2}}} .
\end{align*}
The second term is equal to 
\begin{align*}
-\ob^{2}(x) \oa\b^{2}(1-x) \frac{\pc{\e}{\b^{2}}}{\pc{\oa}{\b^{2}}}.
\end{align*}
Since the above ${}_{3}F_{2}(1)$ in the first term is Saalsch{\"u}tzian and $\g\ob +\g^{2} + \oeta \neq \e + \g\b + \g^{2}\ol{\b^{2} \eta}$ under our assumptions, we have
\begin{align*}
{}_{3}F_{2} \left( 
\begin{matrix}
\g\ob, \g^{2}, \oeta \\
\g\b, \g^{2} \ol{\b^{2}\eta}
\end{matrix}
;  1
\right)
= \frac{\p{\og\b}{\eta} \p{\b^{2}}{\eta}  }{\pc{\g\b}{\eta} \p{\og^{2}\b^{2}}{\eta} } + \frac{\gc{\g\b} \gc{\g^{2}\ob^{2}} }{ g(\g\ob) g(\g^{2}) } \frac{\pc{\e}{\eta}}{\p{\b^{2}\og^{2}}{\eta}}. 
\end{align*}
Therefore 
\begin{align*}
&\frac{1}{1-q} 
\oa (1-x) \sum_{\eta \in \widehat{\kappa^{*}}} \frac{\p{\a + \b^{2}\og^{2} }{\eta}  }{\pc{\b^{2} + \e }{\eta} } \eta \left( \frac{x}{x-1} \right) 
{}_{3}F_{2} \left( 
\begin{matrix}
\g\ob, \g^{2}, \oeta \\
\g\b, \g^{2} \ol{\b^{2}\eta}
\end{matrix}
;  1
\right) \\
&= \oa (1-x) {}_{3}F_{2} \left( 
\begin{matrix}
\a, \b\og, \b^{2} \\
\b^{2}, \b\g
\end{matrix}
;  \frac{x}{x-1}
\right)
+ \oa (1-x) \frac{\gc{\g\b} \gc{\g^{2}\ob^{2}} }{ g(\g\ob) g(\g^{2}) }F\left( \a, \b^{2} ; \frac{x}{x-1} \right)  \\
&= \oa (1-x) {}_{2}F_{1} \left( 
\begin{matrix}
\a, \b\og \\
\b\g
\end{matrix}
;  \frac{x}{x-1}
\right)
+ \oa (1-x) q^{-1} \frac{\p{\a + \b\og}{\ob^{2}}}{\pc{\e + \b\g}{\ob^{2}}} \ob^{2} \left( \frac{x}{x-1} \right) \\
&+ \oa (1-x) \frac{\gc{\g\b} \gc{\g^{2}\ob^{2}} }{ g(\g\ob) g(\g^{2}) }F\left( \a, \b^{2} ; \frac{x}{x-1} \right) .
\end{align*}
Here we used \eqref{cancellation} in the last equality. 
By \eqref{reflection} (resp. \eqref{shift} and \eqref{geom}), one can show that the second (resp. third) term agrees with 
\begin{align*}
\ob^{2} (x) \oa\b^{2} (1-x) \frac{\pc{\e}{\b^{2}}}{\pc{\oa}{\b^{2}}}. \hspace{7mm}\left( resp. \hspace{1mm}\ob^{2} (x) \frac{\p{\ob\g}{\b^{2}} \pc{\e}{\b^{2}} }{\p{\og^{2}}{\b^{2}} \pc{\oa}{\b^{2}}}. \right)
\end{align*}
This proves the formula.
\end{proof}

\begin{remk}
Over the complex numbers, we have the trivial reduction formula
\begin{align*}
F_{1:0:0}^{2:0:0} \left( \left. 
\begin{matrix}
a,  c \\
b
\end{matrix}
;  \hspace{1mm} ;  \hspace{1mm} \right| x ,  y
\right) = \gh{a,  c}{b}{x + y},
\end{align*}
which follows from the binomial theorem. For any $x, y \in \kappa$ with $xy \neq 0$,  the following analogous formula holds:
\begin{align*}
F_{1:0:0}^{2:0:0} \left( \left. 
\begin{matrix}
\a,  \g \\
\b
\end{matrix}
;  \hspace{1mm} ;  \hspace{1mm} \right| x ,  y
\right) = \gh{\a,  \g}{\b}{x + y} +  \delta (x+y). 
\end{align*}
This follows from the following computation:  By putting $\n\m = \eta$ and using \eqref{prelection},  we have
\begin{align*}
F_{1:0:0}^{2:0:0} \left( \left. 
\begin{matrix}
\a,  \g \\
\b
\end{matrix}
;  \hspace{1mm} ;  \hspace{1mm} \right| x ,  y
\right)
&= 
\frac{1}{(1-q)^{2}} 
\sum_{\n,  \m \in \widehat{\kappa^{*}}}  \frac{\p{\a}{\n\m} \p{\g}{\n\m} }{\pc{\b}{\n\m} \pc{\e}{\n} \pc{\e}{\m} } \n (x) \m (y) \\
&= \frac{1}{1-q} \sum_{\eta \in \widehat{\kappa^{*}}}
\frac{\p{\a + \g }{\eta}  }{\pc{\b + \e }{\eta}} 
\eta (x)  {}_{1}F_{0} \left(
\begin{matrix}
\oeta \\
\hspace{1mm}
\end{matrix}
; - \frac{y}{x} \right).
\end{align*}
If we use \eqref{geom},  the right hand side is equal to 
\begin{align*}
&\frac{1}{1-q} \left( \sum_{\eta \neq \e} \frac{\p{\a + \g }{\eta}  }{\pc{\b + \e }{\eta}  } 
\eta \left( x+ y \right) 
 -q \delta \left( 1 + \frac{y}{x} \right) +1   \right) \\
&= \frac{1}{1-q} \left( 
\sum_{\eta \in \widehat{\kappa^{*}} } \frac{\p{\a}{\eta} \p{\g}{\eta}  }{\pc{\b}{\eta} \pc{\e}{\eta} } 
\eta \left( x+ y \right) 
- \e (x+y) - q\delta \left( x+y \right)  + 1 \right)  \\
&= \gh{\a,  \g}{\b}{x + y} + \delta (x+y).
\end{align*}
\end{remk}

As before, by specializing parameters and arguments in Corollary \ref{fifthcor}, we have the following analogous formulas of \cite[Corollary 5.7]{liuwang}.

\begin{cor}
Assume $p \neq 2$.
\begin{enumerate}
\item Suppose that $(\a^2,\b+\g^2)=(\b,\g^2+\e)=(\g^2+\phi\a\ob\g,\e)=0$. Then, we have
\begin{align*}
&\kF{2:0:1}{1:0:1}{\a^2,\g^2}{\b}{}{}{\phi\a\g\ob}{\phi\a\g}{\dfrac{1}{2},-\dfrac{1}{2}}\\
&=\a(4)\sum_{\chi^2=\e}\dfrac{\gc{\phi\a\g}g(\a\chi)}{g(\a^2)\gc{\phi\g\chi}}+\b(-2)\dfrac{\pc{\phi\a\g\ob +\e}{\b}}{\pc{\oa^2+\og^2}{\b}}.
\end{align*}

\item Suppose that $(\a^2,\b + \g + \g^{2} + \b\g + \e) = (\b, \g + \e) = (\g, \e)=0$. Then, we have
\begin{align*}
&\kF{2:0:1}{1:0:1}{\a^2,\g}{\b}{}{}{\a^2\overline{\b\g}}{\a^2\og}{-1,1}\\
&=\oa(4)\sum_{\chi^2=\e}\dfrac{g(\a^2\og) g(\phi)}{g(\phi\a\chi)g(\a\og\chi)} + \dfrac{\pc{\a^2\overline{\b\g}+\e}{\b}}{\pc{\oa^2+\og}{\b}}.
\end{align*}

\item Suppose that $(\a^2,\b+\g^2+\e)=(\b,\g^2+\e)=(\g^2, \a^{2}\b + \a^4+\e)=0$. Then, we have
\begin{align*}
&\kF{2:0:1}{1:0:1}{\a^2,\g^2}{\b}{}{}{\overline{\a^2\b}\g^2}{\oa^2\g^2}{-1,1}\\
&=\oa(4)\sum_{\chi^2=\e}\dfrac{g(\oa\g) g(\phi\oa\g)}{g(\phi\g\chi)g(\oa^2\g\chi)} + \dfrac{\pc{\overline{\a^2\b}\g^{2}+\e}{\b}}{\pc{\oa^2+\og^2}{\b}}.
\end{align*}
\end{enumerate}
\end{cor}

\begin{proof}
Part (i) (resp. (ii), (iii)) follows from Corollary \ref{fifthcor} and \eqref{kummer} (resp. \eqref{gsecond}, \eqref{bailey}).
%\begin{enumerate}
%\item Using Corollary \ref{fifthcor} and \eqref{kummer}, we obtain (i).
%
%\item Using Corollary \ref{fifthcor} and \eqref{gsecond}, we obtain (ii).
%
%\item Using Corollary \ref{fifthcor} and \eqref{bailey}, we complete the proof.
%\end{enumerate}
\end{proof}

\section*{Acknowledgement}
The authors would  like to thank Noriyuki Otsubo and Takato Senoue for valuable discussions. 
They also would like to thank Shinichi Kobayashi and Noriyuki Otsubo for a lot of helpful comments on a draft version of this paper. 
Finally, they would like to thank the anonymous referees for many helpful comments and suggestions. 
The second author was supported by JSPS KAKENHI Grant Number JP22J21285 and WISE program (MEXT) at Kyushu University.
The third author was supported by JST SPRING, Grant Number JPMJSP2109.
The fourth author was supported by Waseda University Grant for Special Research Projects, Project number 2022C-294.

\end{document}